\documentclass{article}
\usepackage[margin=1in]{geometry}
\usepackage[fleqn]{amsmath}
\usepackage{amsthm,amsfonts,amssymb}
\usepackage{authblk}
\usepackage{srcltx}

\usepackage{multirow}
\usepackage{graphicx}
\usepackage{float}
\usepackage{cite}

\usepackage[labelfont=bf]{caption}
\usepackage{enumitem}

\usepackage{amsthm}
\newtheorem{lem}{Lemma}[section] 
\newtheorem{prop}{Proposition}[section]
\newtheorem{cor}{Corollary}[lem]
\theoremstyle{definition}
\newtheorem{dfn}{Definition}[section]

\numberwithin{equation}{section}
\numberwithin{table}{section}
\numberwithin{figure}{section}

\newcommand{\Y}[1]{{\bf(Y#1)}}

\usepackage{tocloft}

\title{State-Space Representation of Hysteresis Systems\\ 
  Exhibiting the Return Point Memory}

\author{S.E.Langvagen\thanks{Electronic address: sergey.langwagen@gmail.com}}

\affil{\small Chernogolovka, Moscow Region}

\date{December 30, 2016}

\begin{document}

\maketitle

\begin{abstract}
  Application of the minimal state-space realization to hysteresis
  systems is studied. The method allows to construct the space of
  states and establish the state transition rules using the {\em input
    equivalence}, which can be obtained for hysteresis systems basing
  on rate independence and the return point memory.
\end{abstract}

\tableofcontents

\section{Introduction}

The science of hysteresis is a science about systems that
demonstrate similar external behavior, but can have a very different
internal structure. Hysteresis is observed in magnetism,
elastoplasticity, ferroelectricity, superconductivity, and other
branches of science \cite{Bertotti&Mayergoyz2006v3}.

For hysteresis systems, the current output depends on the previous
history of the input,
\begin{equation}\label{hyst-op}
  y(t') = {\cal W}[u](t'),\quad t' \in [t_0, t],
\end{equation}
where $u$ and $y$ are input and output functions of time $t'$; ${\cal
  W}$ is called hysteresis operator.  Currently the most common and
efficient approach to mathematical description of hysteresis uses the
models based on the ensembles of simple hysteresis operators like
relay operator, play operator or stop operator
\cite{Krasnoselskii&Pokrovskii1983, Mayergoyz2003,
  Brokate&Sprekels1996ch2, Visintin2006}.

This work studies application of the method known in system theory
as ``minimal state-space realization'' to the hysteresis systems. The
method does not use the decomposition of hysteresis operator or
modeling the internal structure of the system.

The minimal realization is outlined in \cite{Willems1972} as the
``realization that is obtained by considering as the state at time $t$
the equivalence class of those inputs up to time $t$ which yield the
same output after the time $t$ regardless of how the input is
continued after time $t$''.  The minimal state-space realization has
applications in control theory of linear systems
\cite{Polderman&Willems1996ch6}. In finite automata theory, the
similar notion is known as the Nerode equivalence
\cite{Kalman&all1969ch7}. The idea to label metastable states of
hysteresis systems by the field history was proposed in
\cite{Xing2007}.

Let $x(t)$ be a state of the system at time $t$, as it is defined
for the minimal realization. Then, instead of
(\ref{hyst-op}), we have
\begin{equation}\label{fread-out}
  y(t) = f(x(t), u(t)), 
\end{equation}
where the output $y$, the input $u$, and the state $x$ are taken at
the same time instance $t$; $f$ is called {\em read-out function}
\cite{Willems1972}.  The state $x$ comprises information about the
previous history. The subsequent behavior of the system is the same
for inputs that belong to the same equivalence class. Thus, we can
select in each equivalence class one representative input and use
these inputs to introduce coordinates in the space of states. If any
input that belongs to a class $x(t)$ is prolonged up to a time
instance $t + s$, the new input belongs to the equivalence class
$x(t+s)$, which depends on $x(t)$ and on the input $u$ in the time
interval tween $t$ and $t+s$. If we know the coordinates of the old
state, we can get the coordinates of the new state, i.e., the input
equivalence determines the state transition law (see
Appendix~\ref{app:state-space}).

The representation of states as the equivalence classes is well suited
for hysteresis systems, because the needed equivalence relation can be
determined by two well known characteristic properties of hysteresis
--- the rate independence and the return point memory, also known as
wiping out \cite{Mayergoyz2003, Bertotti1998, Brokate&Sprekels1996ch2,
  Visintin2006}. In this article, the consideration is restricted to
the magnetic hysteresis in order to include the demagnetized state
into the scope.

\section{Rate Independence and the Return Point Memory}
\label{sec:RPM}

Consider the behavior of a hysteresis system under slow varying inputs
$H(t')$, $t_0 \leq t'\leq t$, where $H$ is the magnetizing field;
beginning and end times $t_0$ and $t$ may differ for different inputs.
Let the experiments be performed as follows:
\begin{enumerate}[label=(\roman*)]

\item The system is put into the demagnetized state before the beginning
  of each experiment $t_0$;

\item Admissible inputs ${\cal U}^*$ are continuous piecewise-linear
  functions of time $H(t')$, $t'\in(-\infty,t]$, with a finite number
  of segments. It is assumed that $H(t') = 0$ before $t_0$, and
  $|H(t')| \leq H_{max}$ for all inputs;

\item A set $Y$ of output variables is selected. During each
  experiment, the variables included in this set are measured.

\end{enumerate}
For the simplicity, only one output variable $y$ will be considered,
which can be any variable in $Y$.  It seems to be reasonable to expect
that $y$ can represent not only the magnetization $M$ but also other
macroscopic characteristics of the system, such as magnetostrictive
deformation and tension, thermodynamic properties, e.g., free energy,
etc.

Let us denote by $H^t$ the input $H(t')$ with the end time $t$. The
input $H^{t+s}$, $s\geq 0$, is called {\em prolongation} of the input
$H^t$ if $H^{t+s}(t') = H^t(t')$ for all $t'\in (-\infty,t]$. 

It must be clear that, if the final values of the inputs $H^t_1,\,
H^t_2\in{\cal U}^*$ are equal then for any prolongation $H^{t+s}_1$ of
$H^t_1$ there exists a prolongation of the input $H^t_2$ such that
$H^{t+s}_1(t') = H^{t+s}_2(t')$ for all $t'\in [t,t+s]$.  The inputs
$H^{t_1}_1,\, H^{t_2}_2 \in {\cal U}^*$ that have the same final value
of $H$ can be compared in the following way (cf.~Definitions~\ref
{dfn:prolong}, \ref{dfn:*equiv}):

\begin{enumerate}[label=(\roman*)]
\item Shift any of the inputs $H^{t_1}_1,\, H^{t_2}_2$ along the time
  axis, making the end times equal. Denote the shifted inputs by
  $H_1^t,\, H_2^t$;
\item Consider ``equal'' prolongations $H_1^{t+s}, H_2^{t+s}$ such
  that $H_1^{t+s}(t')= H_2^{t+s}(t')$ for all $t'\in[t,t+s]$, and compare
  the outputs $y^{t+s}_1(t'),\,y^{t+s}_2(t')$ in the time interval
  $[t, t+s]$;
\item If $y^{t+s}_1(t') = y^{t+s}_2(t')$ for all $t'\in[t,t+s]$ and
  for any ``equal'' prolongations $H_1^{t+s}(t'), H_2^{t+s}(t')$ then
  we say that the inputs $H^{t_1}_1,\, H^{t_2}_2$ are equivalent and
  write $H^{t_1}_1 \sim H^{t_2}_2$.
\end{enumerate}

The inputs that have different final values of $H$, i.e., such that
$H_1^{t_1}(t_1) \neq H^{t_2}_2(t_2)$, can not be compared and hence
can not be equivalent. The equivalence relation introduced above is
the binary equivalence relation (see Proposition
\ref{prop:*equiv}). Hence, it partitions the set of
admissible inputs into equivalence classes.

In the sequel, we understand the term ``state'' as the class of
equivalent inputs, using as interchangeable the terms ``the inputs are
equivalent'', ``the final states of the inputs are the same'', ``the
inputs belong to the same state''.

\begin{lem}\label{lem:h-prolong}
  Let $H_1^{t+s}, H_2^{t+s}$ be prolongations of equivalent inputs
  $H^t_1 \sim H^t_2$. If $H_1^{t+s}(t') = H_2^{t+s}(t')$ for all
  $t'\in[t,t+s]$, then $H_1^{t+s}\sim H_2^{t+s}$.
\end{lem}
\begin{proof}
  The statement obviously follows from the above definition of
  equivalent inputs (cf. Lemma \ref{lem:sprolong}).
\end{proof}

For the piecewise-linear inputs $H^t\in{\cal U}^*$, the {\em rate
  independence} of the hysteresis system can be described as
follows. If the input changes linearly from the initial demagnetized
state, the final state depends on the final value of $H$ and does not
depend on the slope. The same is true for the second segment of the
input, and so on. This means that the following proposition holds
true:

\begin{prop}\label{prop:rate-indep}
  For rate-independent systems any input $H^t\in{\cal U}^*$ is
  equivalent to the piecewise-linear input $\tilde{H}^t\in{\cal U}^*$
  with positive and negative slopes alternating after $t_0$.  Inputs
  that have the same sequences of local maxima and minima are
  equivalent.
\end{prop}

The other remarkable property typical to many hysteresis systems is
the return point memory (RPM), which is often considered as a property
of the Preisach model. Using compact definition given in
\cite{Sethna&all1993}, RPM can be expressed in terms of the input
equivalence; in this form RPM is completely independent of any
hysteresis model and can be considered as a property of the hysteresis
operator (\ref{hyst-op}), or as a description of the experimental
behavior of the system. This behavior is essentially the same as
expressed by the Madelung rules, noticed over a century ago
\cite{Madelung1905ger}, see also \cite{Brokate&Sprekels1996ch2}. The
return point memory is not precise due to the so-called accommodation
\cite{Bertotti1998}, but, in many cases, the disagreements can be
considered as not very significant.  The definition given in
\cite{Sethna&all1993} is presented below in a slightly changed form as
a proposition.

\begin{figure}[H]
  \begin{center}
    \includegraphics[scale=0.6]{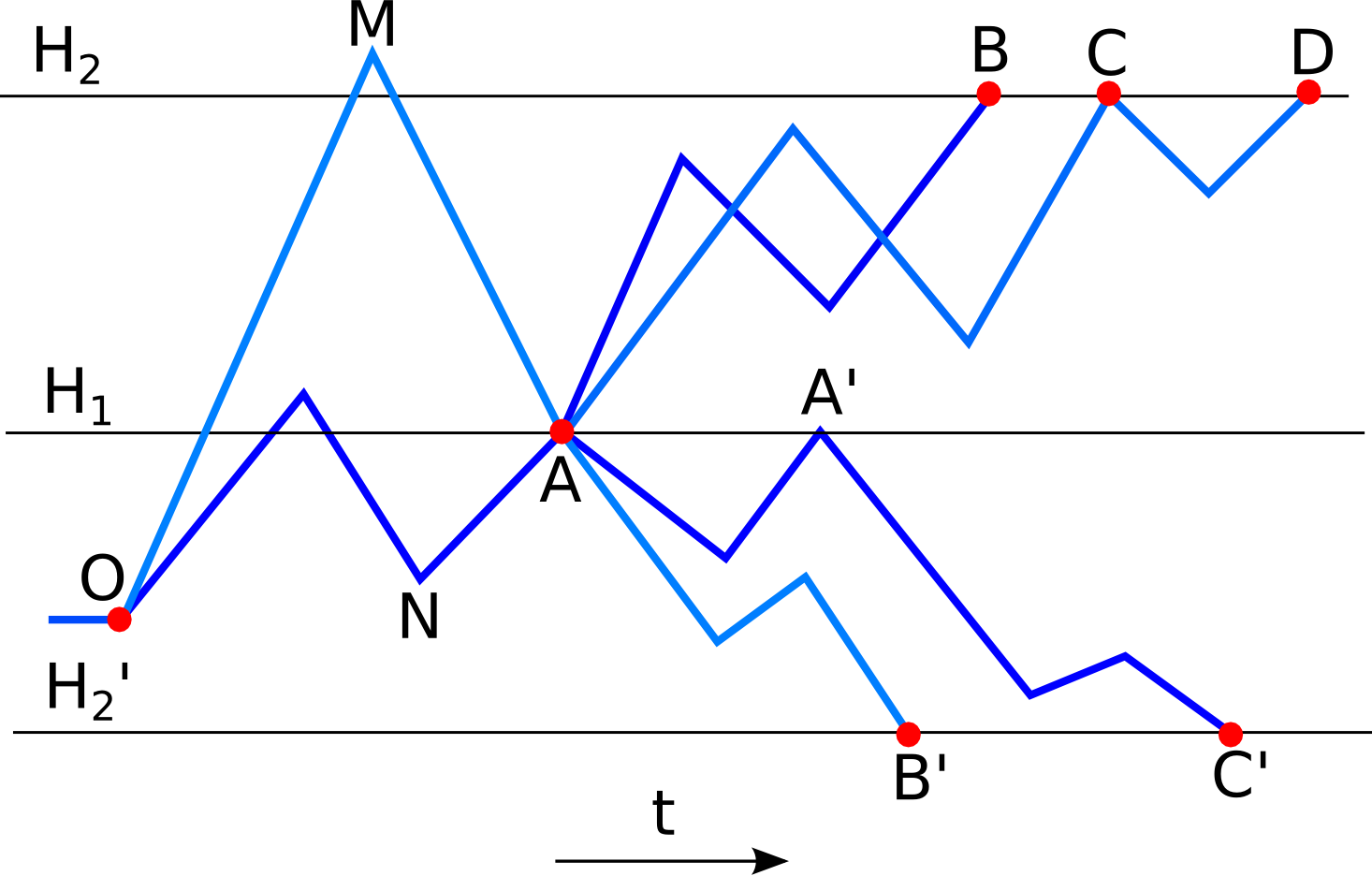}
    \caption{The state at point $O$ is the demagnetized
      state. According to Proposition \ref{prop:RPMI}, the input
      equivalence can be represented as follows: $OMAB\sim OMAC\sim
      OMAD$, $ONAB\sim ONAC\sim ONAD$, and $OMAB'\sim OMAC'$,
      $ONAB'\sim ONAC'$.  If $OMA\sim ONA$, then, due to Lemma
      \ref{lem:h-prolong}, $OMAB\sim OMAC\sim OMAD\sim ONAB\sim
      ONAC\sim ONAD$ and $OMAB'\sim OMAC'\sim ONAB'\sim ONAC'$. The
      states at points $B$, $C$, $D$ are the same for any fixed state
      at point $A$. The same is true for the states at points $B'$,
      $C'$. The states at points $A$, $A'$ are different in a general
      case.}
    \label{fig:rpm}
  \end{center}
\end{figure}

\begin{prop} [{\em Return point memory}] 
  \label{prop:RPMX}
  Let the system is evolved from state $x(t_1)$ under the field
  $H(t')$, $t' \in[t_1, t_2]$, and $H(t_1) = H_1$. Then the state
  $x(t_2)$ depends only on the field $H(t_2) = H_2$, regardless of how
  the field $H(t')$ changed, provided that $H(t')\in [H_1, H_2]$ for
  all $t'\in[t_1,t_2]$.
\end{prop}

Interpreting the states as the equivalence classes and taking into
account Lemma~\ref{lem:h-prolong} and
Proposition~\ref{prop:rate-indep}, we can reformulate the Proposition
\ref{prop:RPMX} in terms of the equivalent inputs.

\begin{prop}[{\em RPM in terms of the input equivalence}]
  \label{prop:RPMI}
  Let input $H^t\in {\cal U}^*$, $H^t(t) = H_1$, be prolonged in two
  different ways by inputs $H_1^{t+s_1}$, $H_2^{t+s_2}$, such that
  $H_1^{t+s_1}(t+s_1) = H_2^{t+s_2}(t+s_2) = H_2$.  If
  $H_1^{t+s_1}(t') \in [H_1, H_2]$ for all $t'\in[t,t+s_1]$ and
  $H_2^{t+s_2}(t')\in [H_1, H_2]$ for all $t'\in [t,t+s_2]$, then
  $H_1^{t+s_1} \sim H_2^{t+s_2}$.
\end{prop}

In Propositions \ref{prop:RPMX}, \ref{prop:RPMI}, $[H_1, H_2]$ denotes
$[min\{H_1,H_2\}, max\{H_1,H_2\}]$.

\begin{dfn}\label{dfn:RPM}
  We say that the system {\em exhibits RPM} if Proposition
  \ref{prop:RPMI} holds true.
\end{dfn}

The return point memory is illustrated by Fig~\ref{fig:rpm}. For the
input $OMACD$ or $ONACD$, the states at points $C$ and $D$ are the
same, and part $CD$ of the input can be omitted, because it does not
influence the final state. Using the terminology of
\cite{Brokate&Sprekels1996ch2}, removing of the redundant part from
the history $H(t)$ will be called the {\em Madelung deletion}.

\section{Reachability of Demagnetized State}\label{sec:demag}

Demagnetized state is the state obtained by applying oscillating field
of amplitude slowly decreasing from a large initial value $H_m$ to
zero \cite{Bertotti1998}, as illustrated by Fig. \ref{fig:demag}.  The
demagnetization must be {\em symmetric}, i.e., all the successive
turning points, such as points $A$ and $B$, must have the same or
approximately the same absolute value of the field $H$.  ``Slowly
decreasing'' means that the amplitudes of adjacent cycles differs by a
small value $\varepsilon$.  For the simplicity, it is assumed that
$\varepsilon$ is constant during the demagnetization. The value of the
output $y$ in the demagnetized state $O'$ is the limiting value at the
end point of the demagnetization process as $\varepsilon \rightarrow
0$.

\begin{figure}[H]
  \begin{center}
    \includegraphics[scale=.5]{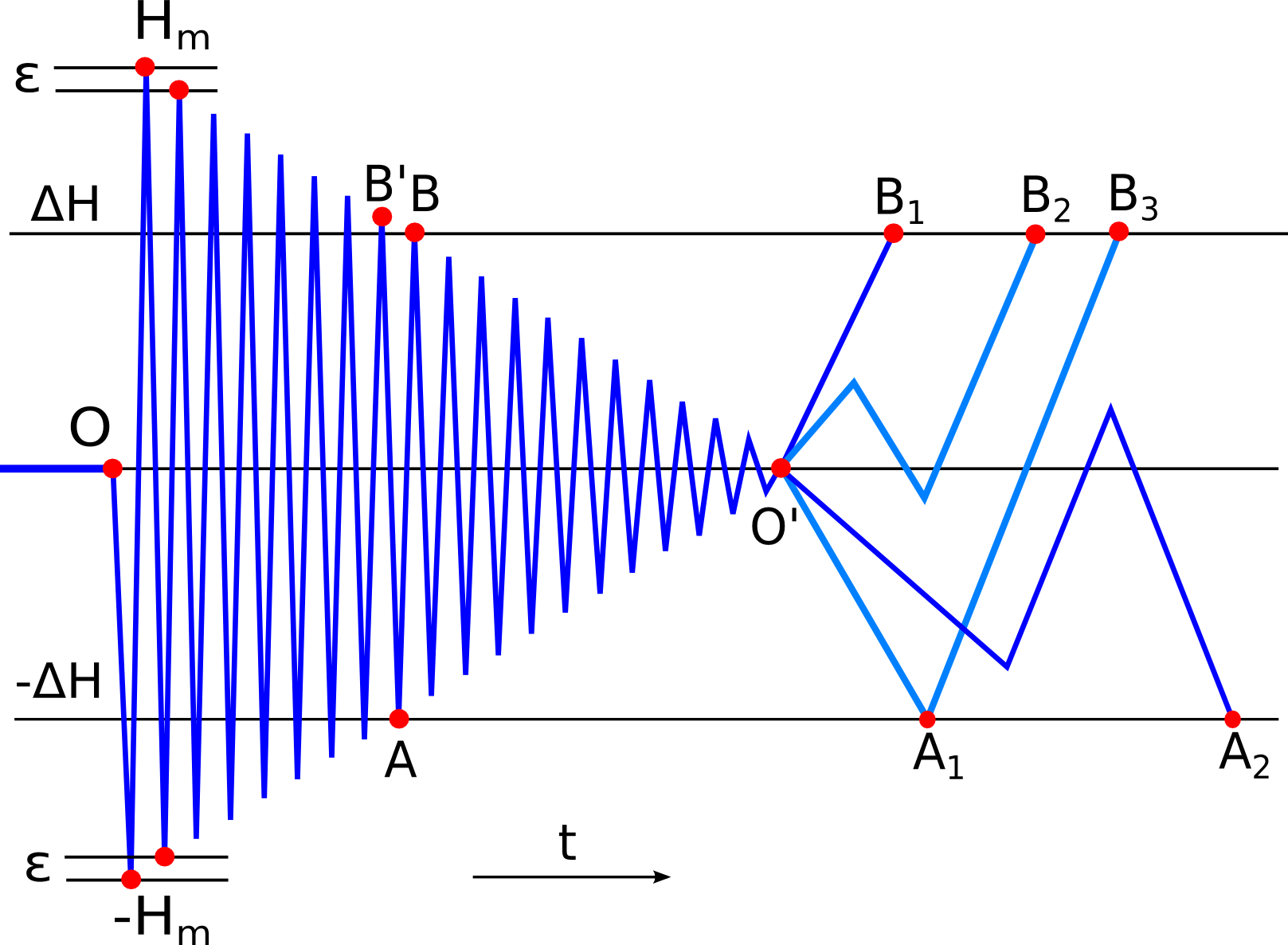}
    \caption{The demagnetized state $O'$ is obtained by the
      demagnetization process $OO'$ as $\varepsilon\rightarrow 0$. 
      Inputs $O'A_1$ and $O'B_1$ correspond to the descending and
      ascending initial magnetization curves respectively. The
      equivalent inputs are: $O'B_1\sim O'B_2\sim O'A_1B_3$ and
      $O'A_1\sim O'A_2$.}
    \label{fig:demag}
  \end{center}
\end{figure}

\begin{dfn}\label{dfn:rdemag}
  Let us say that the {\em demagnetized state is reachable}, if
  the state obtained by the described above demagnetization process,
  as $\varepsilon\rightarrow 0$, is indistinguishable from the initial
  demagnetized state.
\end{dfn}

Applying Proposition \ref{prop:RPMI} to the input $OA$ and
prolongations $OAB$, $OABO'B_1$, $OABO'B_2$, $OABO'A_1B_3$
(Fig. \ref{fig:demag}), it can be seen that $OAB\sim OABO'B_1\sim
OABO'B_2\sim OABO'A_1B_3$, i.e., the states at points $B$, $B_1$,
$B_2$, $B_3$ are the same. In the similar way, examining the input
$OB'$ and prolongations $OB'A$, $OB'AO'A_1$, $OB'AO'A_2$, it can be
found that the states at points $A$, $A_1$, $A_2$ are the same.
Considering the state at point $O'$ as the initial demagnetized state,
we can write $O'B_1\sim O'B_2\sim O'A_1B_3$ and $O'A_1\sim O'A_2$.
Similar results can be obtained for any input that starts from the
point $O'$ and lays inside the region $[-\Delta H, \Delta H]$.
This gives the following proposition.

\begin{prop}\label{prop:demag-equiv}
  Let the system exhibits RPM and has reachable demagnetized state.
  If the system is put into the demagnetized sate at the time $t_0$,
  $H^t\in {\cal U}^*$, and $H^t(t')\in [-\Delta H,\Delta H]$ for all
  $t'\in[t_0, t]$, then all $H^t$, such that $H^t(t) = \Delta H$ are
  equivalent, and all $H^t$, such that $H^t(t) = -\Delta H$, are
  equivalent.
\end{prop}

\begin{cor}\label{cor:demag-equiv}
  The state obtained from the demagnetized state after increasing
  (decreasing) the field by $\Delta H$ can also be reached from the
  demagnetized state via decreasing (increasing) the field by $\Delta
  H$ and then increasing (decreasing) it by $2\Delta H$.
\end{cor}

Starting from the demagnetized state $O$, we can return back to this
state, e.g., decreasing the field by the value $\Delta H$, not
necessary large, and then performing the demagnetization process $AO'$
(Figure \ref{fig:pdemag}). In this way we merely continue the previous
demagnetization performed with a large initial amplitude.

If the system has the reachable demagnetized state according to
Definition \ref{dfn:rdemag}, then, for sufficiently small
$\varepsilon$, the states $O$ and $O'$ can be considered as identical.
Thus, the output $y$ at the end of any input applied to the state
$O'$, as $\varepsilon\rightarrow 0$, must tend to the value $y$ at the
end of the same input applied directly to the state $O$. As an
example, the value $y$ at the end point of input $OAO'B'CD$ must tend
to the value $y$ at end point of input $OB'CD$.

Consider the input $OABO'B'$. The states at points $B$, $B'$, as it
was shown before, are the same. Because of this, the part $BO'B'$ can
be omitted, which means that $OABCD\sim OAO'B'CD$, i.e., the input
with ``full'' demagnetization $AO'$ can be replaced by the input with
``partial'' demagnetization $AB$.  Thus, the output $y$ at the end of
the input $OABCD$ must tend to the output at the end of the input
$OB'CD$, as $\varepsilon\rightarrow 0$. The condition of this kind
allows to impose on the read-out functions the restriction that
expresses the reachability of demagnetized state (see
Section~\ref{sec:read-out} later on).

\begin{figure}[H]
  \begin{center}
    \includegraphics[scale=.6]{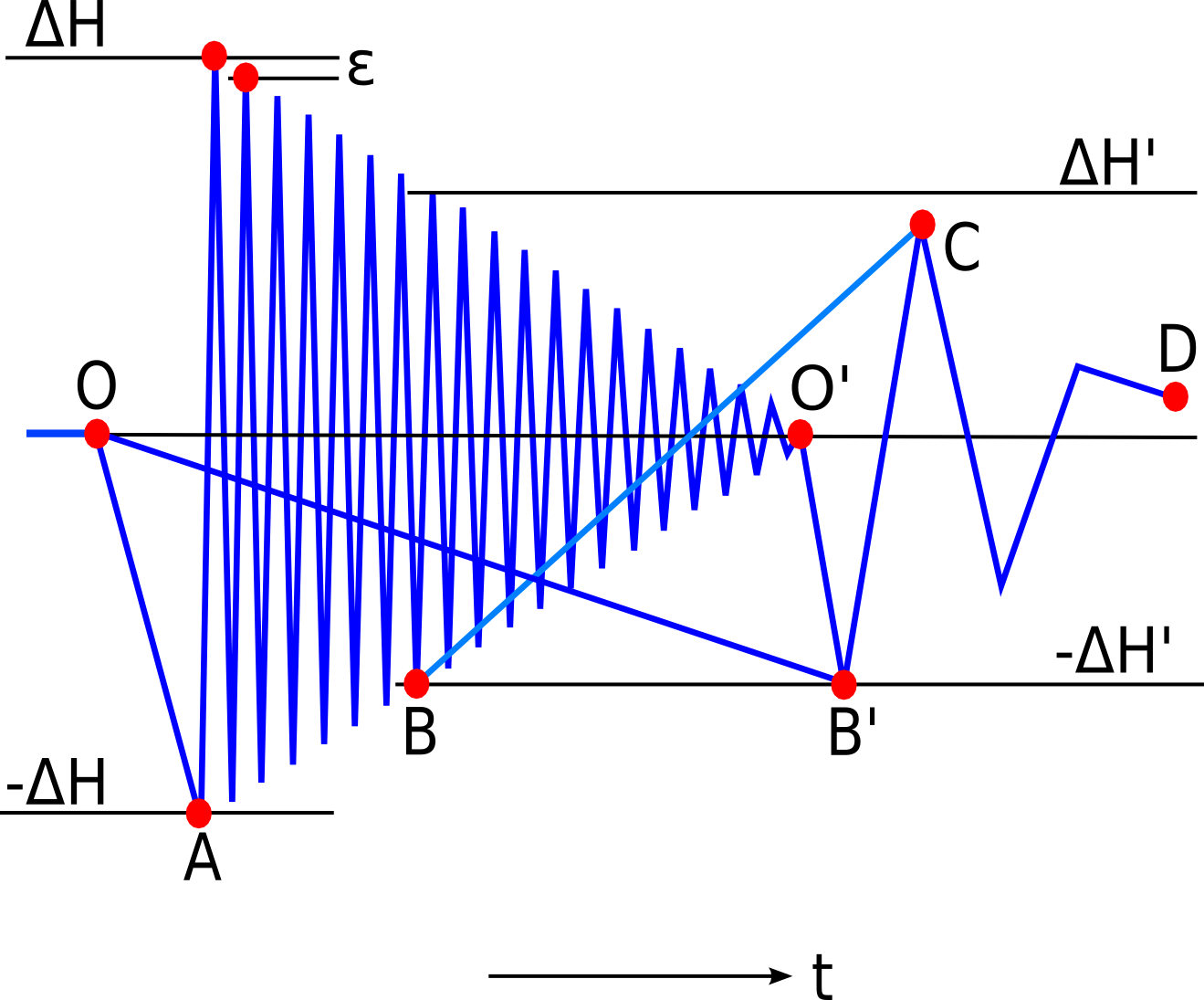}
    \caption{ Inputs $OAO'B'CD$ and $OABCD$ are equivalent. As
      $\varepsilon \rightarrow 0$, the output value $y$ at the end of
      input $OABCD$ must tend to the value $y$ at the end of input
      $OB'CD$.}
    \label{fig:pdemag}
  \end{center}
\end{figure}

\section{The Space of States and the State Transition Rules}
\label{sec:space_of_states}

According to Proposition \ref{prop:rate-indep}, the inputs
$\tilde{H}^t\in{\cal U}^*$ that have the same sequence of local maxima
and minima $\tilde{H}_0\tilde{H}_1\,\ldots\,\tilde{H}_k$ are
equivalent.  We can reduce the number of elements in the sequence by
replacing successively the inputs with equivalent ones as follows.
Let $\tilde{H}_{i_0}$ be the element in the sequence that has the
maximum absolute value, or, if there are more than one such element,
the last one. Assume at first that $\tilde{H}_{i_0}<0$ and let
$H_0=\tilde{H}_{i_0}$. As follows from
Proposition~\ref{prop:demag-equiv} and Lemma ~\ref{lem:h-prolong},
$\tilde{H}_0\tilde{H}_1\,\ldots\,\tilde{H}_k \sim
H_0\tilde{H}_{i_0+1}\tilde{H}_{i_0+2}\,\ldots\,\tilde{H}_k$. Let
$\tilde{H}_{i_1}=H_1$ be the last maximal element in the sequence
$\tilde{H}_{i+1}\tilde{H}_{i+2}\,\ldots\,\tilde{H}_k$. According to
Proposition~\ref{prop:RPMI} and Lemma~\ref{lem:h-prolong},
$\tilde{H}_0\tilde{H}_1\,\ldots\,\tilde{H}_k\sim
H_0H_1\tilde{H}_{i_1+1}\tilde{H}_{i_1+2}\,\ldots\,\tilde{H}_k$. In the
similar way we have $\tilde{H}_0\tilde{H}_1\,\ldots\,\tilde{H}_k\sim
H_0H_1H_2\tilde{H}_{i_2+1}\tilde{H}_{i_2+2}\,\ldots\,\tilde{H}_k$,
where $\tilde{H}_{i_2} = H_2$ is the last minimal element in the
sequence $\tilde{H}_{i_1+1}\tilde{H}_{i_1+2}\,\ldots\,\tilde{H}_k$.
Selecting maximum and minimum elements alternatively, we end up with
the {\em reduced memory sequence} \cite{Visintin2006}
$H_0H_1H_2\,\ldots\,H_n$ that is equivalent to the initial input
$\tilde{H}_0\tilde{H}_1\,\ldots\,\tilde{H}_k$.

In the case $\tilde{H}_{i_0}>0$, let $H_0 = -\tilde{H}_{i_0}$ and
$H_1=\tilde{H}_{i_0}$. Taking into account Corollary
\ref{cor:demag-equiv}, we have
$\tilde{H}_0\tilde{H}_1\,\ldots\,\tilde{H}_k\sim
H_0H_1\tilde{H}_{i_0+1}\tilde{H}_{i_0+2}\,\ldots\,\tilde{H}_k$. Now
the remaining elements in the sequence can be defined as before,
getting the reduced memory sequence
$\tilde{H}_0\tilde{H}_1\,\ldots\,\tilde{H}_k \sim H_0H_1\,\ldots\,H_n$
with $H_0<0$. In this way, we excluded the sequences with different
signs of $H_0$.

The last case, $\tilde{H}_{i_0} = 0$, corresponds to the trivial
input, such that $\tilde{H}^t(t') = 0$ for all $t'\in(-\infty, t]$,
and to the reduced memory sequence with $n = 0$ and $H_0 = 0$.

The results of the above consideration can be represented as the
following lemma:
\begin{lem}\label{lem:rmem-seq}
  For systems that exhibit RPM and have reachable demagnetized state,
  any input $\tilde{H}^t\in{\cal U}^*$ is equivalent to the input
  $H^t\in{\cal U}^*$ determined by the reduced memory sequence
  $H_0,H_1,\,\ldots\,,H_n$, with $n=0,1,\ldots$, such that
  \begin{equation}\label{Hi-ineq}
    -H_{max}\leq H_0\leq 0, \quad H_1 \leq |H_0|,\quad H_0 < H_2 < H_4 \ldots\,,
    \quad H_1 > H_3 > H_5 \ldots\,,
    \quad H = H_n,
  \end{equation}
  where $H_1,H_3,H_5,\,\ldots$ are local maxima and $H_0,H_2,H_4,\,\ldots$ 
  are local minima of the input $H^t$.
\end{lem}
Because the input equivalence is transitive, all the inputs that are
equivalent to the input $H_0H_1\,\ldots\,H_n$ are equivalent to each
other. Thus, the variables $H_0,H_1,\,\ldots\,,H_n$ determine the
class of equivalent inputs, i.e., the state of the system. This means
that the reduced memory sequences $H_0,H_1,\,\ldots\,,H_n$ with
$H_0\leq0$ and $n = 0,1,\,\ldots\,$, represent coordinates in the
state-space of a system that exhibits RPM and has the reachable
demagnetized state according to Definitions~\ref{dfn:RPM},
\ref{dfn:rdemag}.

As usual, different coordinate systems can be used for parametrization
of the state-space. Consider another coordinates introduced below,
which seem to be more convenient to represent the read-out functions.
Let us define
\begin{equation}\label{DH-def}
\Delta H_0= |H_0|,\;\Delta H_1= |H_1-H_0|,\;\ldots\;,\;
\Delta H_n=|H_n-H_{n-1}|.
\end{equation}

\begin{lem}\label{lem:r-input}
  Any input $\tilde{H}^t\in{\cal U}^*$ is equivalent to the input
  $H^t\in{\cal U}^*$ such that, starting from the demagnetized state
  at $H = 0$, the field linearly decreases by the value $\Delta
  H_0$, then increases by $\Delta H_1$ then decreases by $\Delta H_2$,
  and so on till $\Delta H_n$, where $n = 0,\,1,\,2\,\ldots$ and
\begin{equation}\label{DH-ineq}
  2H_{max}\geq 2\Delta H_0 \geq \Delta H_1 > \Delta H_2 > \ldots > \Delta H_n > 0.
\end{equation}
\end{lem}
\begin{proof} Omitted. \end{proof}

Any state reachable from the demagnetized state, can be obtained by
the input $\Delta H_0\Delta H_1\,\ldots\,\Delta H_n$ according to
Lemma \ref{lem:r-input}.  To shorten the notations, let us
introduce the variables $\xi_0,\xi_1\,\ldots\,,\xi_n$,
\begin{equation}\label{xi-def}
  \xi_0 = 2\Delta H_0,\, \xi_1 = \Delta H_1,\, \xi_2 = 
  \Delta H_2,\, \ldots,\, \xi_n = \Delta H_n.
\end{equation}
It is convenient to accept instead of (\ref{DH-ineq}) less strict
inequalities
\begin{equation}\label{xi-ineq}
  \xi_M \geq \xi_0 \geq \xi_1 \geq \xi_2 \geq \ldots \geq \xi_n \geq 0,
\quad \mbox{where}\quad \xi_M/2 = H_{max}.
\end{equation}
Taking into account that the signs of $H_i - H_{i-1}$ are alternating
in (\ref{DH-def}), the reduced memory sequence
$H_0,H_1,\,\ldots\,,H_n$ can be expressed via
$\xi_0,\xi_1\,\ldots\,,\xi_n$ as follows:
\begin{equation}\label{xi->rmem}
  H_k = -\frac{1}{2}\xi_0 + \xi_1 - \ldots \pm\xi_k,
  \quad\mbox{where}\quad k = 0,1,2,\ldots\,\,n,
\end{equation}
and for the field $H$ we have
\begin{equation}\label{H(xi)}
  H = -\frac{1}{2}\xi_0 + \xi_1 - \ldots \pm\xi_n.
\end{equation}

The equations (\ref{xi->rmem}) determine a linear reversible transformation between
variables $\xi_0,\ldots,\xi_n$ and $H_0,\ldots,H_n$, which means that
$\xi_0, \ldots, \xi_n$ are equally acceptable as coordinates as the
reduced memory sequences $H_0,\ldots,H_n$.

\medskip

In the coordinates $\xi_0,\xi_1\, \ldots\,,\xi_n$, the
demagnetized state is $(0)$. The hysteresis branch that corresponds to
$\xi_n$ is ascending for odd $n$ and descending for even $n$.
The state $(\xi_0,\,\ldots\,,\xi_{n-1}, 0)$ is obviously the same as
the state $(\xi_0,\,,\ldots\,,\xi_{n-1})$.  If $\xi_i = \xi_{i+1}$ and
$1\leq i\leq n-1$, the state $(\xi_0,\,\ldots\,,\xi_i,\xi_{i+1},\,
\ldots\,,\xi_n)$ is the same as the state $(\xi_0,\,\ldots,\,
\xi_{i-1},\xi_{i+2},\,\ldots\,\xi_n)$ according to Proposition
\ref{prop:RPMI} and Lemma~\ref{lem:h-prolong}.  Excluding equal
adjacent coordinates corresponds to the {\em Madelung deletion}
mentioned in the previous section.

Let us examine how the state variables evolve when the input changes
(cf. Propositions \ref{prop:state-trans}, \ref{prop:semi-group}).
Let $H$ increases or decreases by the small value $\delta H$, not
violating the inequalities (\ref{xi-ineq}). 

From the initial demagnetized state $(0)$ we can go along the
descending or ascending magnetization curves, and the new state will
be $(2|\delta H|)$ if $\delta H < 0$ or $(2\delta H, 2\delta H)$ if
$\delta H > 0$.

The state $(\xi_0, \xi_0)$ is the state on the ascending magnetization
curve with $H = \xi_0/2$. If $\delta H > 0$, the new state is on the
same curve with $H = \xi_0/2 + \delta H$, which is the state $(\xi_0 +
2\delta H, \xi_0 + 2\delta H)$.

For other states $(\xi_0, \ldots, \xi_n)$ with odd $n$, if $\delta H <
0$, the variable $\xi_{n+1} = |\delta H|$ is added, because $H(t)$
starts to decrease after increasing.  Otherwise, if $\delta H > 0$,
$H(t)$ continues increasing, and the new state is $(\xi_0, \ldots,
\xi_n + \delta H)$.  The similar is true for even $n$, with opposite
signs of $\delta H$.

Due to the Madelung deletion that must be performed if $\xi_n$ becomes
equal to $\xi_{n-1}$ for $n \geq 2$, the inequalities (\ref{DH-ineq})
remain true.

The following proposition summarizes the above results.
\begin{prop}
  The sequences of variables $\xi_0,\xi_1\,\ldots\,,\xi_n$, $n =
  0,1,\ldots$, defined according (\ref{xi-def}), (\ref{xi-ineq}), can
  be accepted as coordinates in the state-space of a system that
  exhibits RPM and has reachable demagnetized state. When $H$ changes,
  the coordinates of the the state change according to
  Table~\ref{tab:state-trans}.
\end{prop}

\begin{table}[H]
  \caption{State Transition Rules}  \label{tab:state-trans}
  \renewcommand*{\arraystretch}{1.6}
  \begin{center}
    \begin{tabular}{|l|l|l|}
      \hline
      \multicolumn{1}{|c|}{\multirow{2}{*}
        {\bf Current Sate}}&\multicolumn{2}{|c|}{\bf New state}\\
      \cline{2-3}
      &\multicolumn{1}{|c|}{$ \delta H > 0 $} & \multicolumn{1}{|c|}{$\delta H < 0$}\\
      \hline
      $(0)$ & $(2\delta H,2\delta H)$ & $(2|\delta H|)$\\
      $(\xi_0,\xi_0)$ & $(\xi_0 + 2\delta H,\xi_0 + 2\delta H)$ & 
      $(\xi_0,\xi_0, |\delta H|)$\\
      $(\xi_0,\ldots,\xi_n),\,n\mbox{ even}$ & $(\xi_0,\ldots,\xi_n,\delta H)$ & 
      $(\xi_0,\ldots,\xi_n+|\delta H|)$\\
      $(\xi_0,\ldots,\xi_n),\,n\mbox{ odd}$  & $(\xi_0,\ldots,\xi_n+\delta H)$ & 
      $(\xi_0,\ldots,\xi_n,|\delta H|)$\\
      \hline
    \end{tabular}
  \end{center} {\em Notes:}
  \begin{enumerate}[label=(\roman*)]
  \item $\delta H$ must be small enough for the new state to be
    in agreement with (\ref{xi-ineq}).
  \item If $\xi_n$ becomes equal to $\xi_{n-1}$, $n \geq 2$, the new
    state will be $(\xi_0,\ldots,\xi_{n-2})$.
  \end{enumerate}
\end{table}

\section{Read-Out Functions}\label{sec:read-out}

For a system that exhibits RPM and has reachable demagnetized state,
any output value $y$ that depends on the state of the system can be
expressed as a sequence of functions
\begin{equation}\label{y(xi)-read-out}
  y_n(\xi_0,\ldots,\xi_n),\quad\mbox{where}\quad n = 0,1,\,\dots\,, 
\end{equation}
defined on the region $D_n(\xi_M)$ determined by inequalities
(\ref{xi-ineq}).  Note that the input $H$ does not need to be an
argument of the read-out functions (\ref{y(xi)-read-out}) due to
(\ref{H(xi)}), cf. (\ref{fread-out}).  According to (\ref{xi-def}),
the coordinates $\xi_0,\ldots,\xi_n$ describe how a given state can be
obtained from the initial demagnetized state. As a matter of fact,
functions (\ref{y(xi)-read-out}) represent multiple order reversal
curves in the $H$-$y$ plane with $n+1$ branches.

We further restrict the consideration to systems that exhibit smooth
multiple order reversal curves, assuming that
$y_n(\xi_0,\ldots,\xi_n)$ are sufficiently many times differentiable
on $D_n(\xi_M)$ and have partial derivatives uniformly bounded with
respect to $n$:
\begin{equation}\label{deriv-bounded}
  \left| \frac{\partial^ky_n}{\partial^{k_0}\xi_0,\ldots,\partial^{k_n}\xi_n}\right|
  \leq C_k,\,\mbox{where}\quad k=k_0 + \ldots + k_n.
\end{equation}
Actually (\ref{deriv-bounded}) will be used for $k \leq 2$ only.

As mentioned above, the magnetization $M$ and probably some other
macroscopic physical values, such as thermodynamic potentials,
magnetostrictive deformation, etc., can be expressed by (\ref{y(xi)-read-out}).

Let us consider conditions that must be imposed on functions
$y_n(\xi_0,\ldots,\xi_n)$.  The initial point of $(n+1)$-th hysteresis
branch is the final point of $n$-th branch, which gives
the following condition:
\begin{equation}\nonumber
  \mbox{\Y0}\;\;\; y_n(\xi_0,\xi_1,\ldots,\xi_n) = 
  y_{n-1}(\xi_0,\xi_1,\ldots,\xi_{n-1})$,\;\mbox{ if }$\xi_n = 0,\,n\geq 1.
\end{equation}

\noindent
When $\xi_k = \xi_{k+1}$, the Madelung deletion can be applied,
thus, we have
\begin{flalign*}
  \mbox{\Y1}\;\;\;y_n(\xi_0,\ldots,\xi_k,\xi_{k+1},\ldots,\xi_n) = 
  y_{n-2}(\xi_0,\ldots,\xi_{k-1},\xi_{k+2},\ldots,\xi_n),\\
  \mbox{ if }\xi_k = \xi_{k+1},\,\, 1\leq k\leq n - 1,\, n\geq 2.
\end{flalign*}

\smallskip
\begin{lem}\label{lem:Y1-diff}
  The \Y1 condition has the equivalent form
  \begin{equation}\label{Y1-diff}
    \frac{\partial y_n}{\partial\xi_k}+\frac{\partial y_n}{\partial\xi_{k+1}} = 0,\;
    \mbox{ if }\;\xi_k=\xi_{k+1},\;1\leq k\leq n-1. 
  \end{equation}
\end{lem}
\begin{proof}
  Obviously, (\ref{Y1-diff}) follows from \Y1.  Let us prove
  that (\ref{Y1-diff}) implies \Y1.  When the adjacent pair of
  equal coordinates changes, $y_n$ does not changes due to
  (\ref{Y1-diff}).  If the pair of equal coordinates in the left
  side of \Y1 is not the last one, we may decrease $\xi_k = \xi_{k+1}$
  until $\xi_k = \xi_{k+1} = \xi_{k+2}$, then decrease the pair
  $\xi_{k+1} = \xi_{k+2}$ in the same way, and so on. Eventually we
  get $\xi_{n-1} = \xi_n$. Making this pair equal to zero and using
  \Y0 twice gives the right side of \Y1.
\end{proof}

\begin{lem}\label{lem:lipschitz-plus}
  The inequality
  \begin{equation} \label{lipschitz-plus.1} \left| \frac{\partial
        y_n}{\partial\xi_i} + (-1)^{k+1}\frac{\partial
        y_n}{\partial\xi_{i+k}} \right| \leq C_2\cdot (\xi_i -
    \xi_{i+k})
  \end{equation}
   holds for for any $i\geq 1$, $k\geq 1$, such that $1\leq i+k\leq n$.
\end{lem}
\begin{proof}
  Taking into account Lemma \ref{lem:Y1-diff} and
  (\ref{deriv-bounded}) we can write
  \begin{equation}\label{lipschitz-plus.2}
    \left|\frac{\partial y_n}{\partial\xi_i} + \frac{\partial y_n}{\partial\xi_{i+1}}\right| 
    \leq C_2\cdot (\xi_i - \xi_{i+1}).
  \end{equation}
  The absolute value
  \begin{equation}\nonumber
    \left|
      \left(\frac{\partial y_n}{\partial\xi_i} + \frac{\partial y_n}{\partial\xi_{i+1}}\right)
      -\left(\frac{\partial y_n}{\partial\xi_{i+1}} + \frac{\partial y_n}{\partial\xi_{i+2}}\right)
      +\ldots
      +(-1)^{k+1}\left(\frac{\partial y_n}{\partial\xi_{i+k-1}} + \frac{\partial y_n}{\partial\xi_{i+k}}\right)
    \right|
  \end{equation}
  equals to the left side of (\ref{lipschitz-plus.1}), and, as follows
  from (\ref{lipschitz-plus.2}), is not greater than $C_2\cdot (\xi_i
  - \xi_{i+k})$.
\end{proof}

Conditions \Y0, \Y1 do not guarantee, e.g, that after the
demagnetization process the output value $y$ will be the same as in the
initial demagnetized state.  The output $y$ at the end of any input
applied to the state $O'$ obtained after the demagnetization, as
$\varepsilon\rightarrow 0$, must tend to the output $y$ after applying
the same input to the initial demagnetized sate $O$ (see Figure
\ref{fig:pdemag}). If this is true, we can say that the demagnetized
state is reachable. Taking into account that the ``full''
demagnetization can be replaced by the ``partial'' demagnetization, as
described in Section \ref{sec:demag}, the sufficient and necessary
condition for the reachability of demagnetized state can be written as
follows:
\begin{equation}\label{y-pdemag}
  \lim_{N \to \infty}y_{n+2N}(\xi_0,\xi_0 -\varepsilon,\xi_0 - 2\varepsilon,\ldots,\xi_0 - 2N\varepsilon,
  \tilde{\xi}_1,\ldots,\tilde{\xi}_n) = y_n(\tilde{\xi}_0, \tilde{\xi}_1, \dots, \tilde{\xi}_n),
\end{equation}
where $\varepsilon = (\xi_0 - \tilde{\xi}_0)/2N$, and the variables 
$\tilde{\xi}_1, \ldots, \tilde{\xi}_n$ correspond to an arbitrary process
performed after the ``partial'' demagnetization. 

\medskip
Let us consider the following condition:
\begin{equation}\nonumber
  \mbox{\Y2}\;\;\; 2\frac{\partial y_n}{\partial \xi_0} + \frac{\partial y_n}{\partial \xi_1} = 0,\; 
  \mbox{if}\; \xi_0 = \xi_1,\, n = 1, 2, \ldots\,.
\end{equation}

\begin{lem} Condition \Y2 implies \label{demag-step}
  \begin{equation}\nonumber
    y_n(\xi_0,\xi_0 - \varepsilon,\xi_2, \ldots, \xi_n) =  
    y_n(\xi_0 - 2\varepsilon, \xi_0 - 2\varepsilon,\xi_2, \ldots, \xi_n) + O(\varepsilon^2),\;
    \mbox{as}\; \varepsilon \rightarrow 0,
  \end{equation}
  where the estimate $O(\varepsilon^2)$ does not depend on $n$.
\end{lem}

\begin{proof}
  From the Taylor's theorem we have
  \begin{align*}
    y_n(\xi_0,\xi_0 - \varepsilon,\,\xi_2,\ldots,\xi_n) &
    = y_n(\xi_0,\xi_0,\xi_2,\ldots,\xi_n) - \frac{\partial y_n}{\partial\xi_1}\cdot\varepsilon + O(\varepsilon^2),\\
    y_n(\xi_0 - 2\varepsilon, \xi_0 - 2\varepsilon,\xi_2,\ldots,\xi_n) &=
    y_n(\xi_0,\xi_0,\xi_2,\ldots,\xi_n) - \left(\frac{\partial y_n}{\partial
        \xi_0} + \frac{\partial y_n}{\partial \xi_1}\right)\cdot
    2\varepsilon + O(\varepsilon^2),
  \end{align*}
  where the estimate $O(\varepsilon^2)$ does not depend on $n$ due to
  (\ref{deriv-bounded}).  Subtracting one equation from the other
  and using \Y2 gives the statement of the lemma.
\end{proof}

\begin{lem}\label{lem:Y2->y-demag}
  Condition \Y2 implies that after the demagnetization the output
  value $y$ returns to its value in the initial demagnetized state as follows:
\begin{equation}\label{y-demag}
  \lim_{N \to \infty} y_{2N} ( \xi_0, \xi_0 - \varepsilon, \xi_0 - 2\varepsilon, \ldots, \xi_0- 2N\varepsilon) = y_0(0),
\end{equation}
where $\varepsilon = \xi_0/(2N+1)$, and $N$ denotes the number of
demagnetization cycles.
\end{lem}

\begin{proof}
  The left side of (\ref{y-demag}) can be transformed to
  \begin{equation}\nonumber
    y_{2N} ( \xi_0 -2\varepsilon, \xi_0 - 2\varepsilon, \xi_0 - 2\varepsilon, \ldots, \xi_0 - 2N\varepsilon) + O(\varepsilon^2)
  \end{equation}
  by using Lemma \ref{demag-step}. Applying the Madelung deletion \Y1 to
  the second and third arguments gives
  \begin{equation}\nonumber
    y_{2N-2}(\xi_0-2\varepsilon,\xi_0-3\varepsilon,\ldots,\xi_0-2N\varepsilon)
     + O(\varepsilon^2).
  \end{equation}
  By repeating the same steps $N$ times, finally we get
  \begin{equation}\nonumber
    y_n (\xi_0 -2N\varepsilon) + N\cdot O(\varepsilon^2) = 
    y_n (\varepsilon) + N\cdot O(\varepsilon^2) \rightarrow y_0(0),
  \end{equation}
  taking into account that  $\varepsilon = \xi_0/(2N+1)$.
\end{proof}

\begin{prop}\label{prop:rdemag<->Y2}
  Condition \Y2 is necessary and sufficient for the reachability of
  demagnetized state according to (\ref{y-pdemag}) .
\end{prop}

\begin{proof}
  The proof that \Y2 implies (\ref{y-pdemag}) is almost the
  same as the proof of Lemma \ref{lem:Y2->y-demag} and
  is omitted.  Let us prove that (\ref{y-pdemag}) implies \Y2.
  Rewrite the left side of (\ref{y-pdemag}) as follows:
  \begin{equation}\label{y-pdemag*}
    y_{n+2N}(\xi_0,\xi_1,\dots,\xi_{2N},\tilde{\xi}_1,\ldots,\tilde{\xi}_n),
  \end{equation}
  where \vspace{-3mm}
  \begin{equation}\nonumber
    \xi_0 - \xi_1 = \xi_1 - \xi_2 = \ldots = \xi_{2N-1}-\xi_{2N} = \frac{\Delta\xi}{2N} = \varepsilon,\; \Delta\xi = \xi_0  - \tilde{\xi}_0.
  \end{equation}

  Let us replace in (\ref{y-pdemag*}) the arguments $\xi_i$
  having odd indexes by $\xi_i - \theta\varepsilon$ and consider the
  result as the function of $\theta$
  \begin{equation}\label{y-pdemag*(theta)}
    y(\theta) = y_{n + 2N} (\xi_0, \xi_1-\theta\varepsilon, \xi_2, \xi_3-\theta\varepsilon, \ldots, \xi_{2N-1}-\theta\varepsilon, \xi_{2N}, 
    \tilde{\xi}_1, \ldots, \tilde{\xi}_n).
  \end{equation}
  Note that $y(0)$ equals to (\ref{y-pdemag*}), and
  $ y(1) = y_n(\tilde{\xi}_0 +\Delta\xi,\tilde{\xi}_1,\ldots\tilde{\xi}_n)$
  due to the Madelung deletion of the adjacent arguments $\xi_i$.
  Applying the Taylor's theorem to $y(\theta)$ gives
  \begin{align}
    y_n(\tilde{\xi}_0 + \Delta\xi,\tilde{\xi_1},\ldots,\tilde{\xi}_n) &= 
    y_{n + 2N} (\xi_0,\xi_1,\dots,\xi_{2N},\tilde{\xi}_1,\ldots,\tilde{\xi}_n) \nonumber\\
    &+ \sum^N_{i=1}\frac{\partial y_{n+2N}}{\partial\xi_{2i-1}}\cdot
    \frac{\Delta\xi}{2N}+\frac{1}{2}\sum^N_{i,j=1}\frac{\partial^2 y_{n+2N}}
    {\partial\xi_{2i-1}\partial\xi_{2j-1}}\cdot\left(\frac{\Delta\xi}{2N}\right)^2.
    \label{y-pdemag*(theta)-taylor}
  \end{align}
  Similarly, for the derivative $\partial y(\theta)/\partial\tilde{\xi}_1$
  \begin{equation}\label{y-pdemag*(theta)-diff-taylor}
    \frac{\partial y_n}{\partial\tilde{\xi}_1} = 
    \frac{\partial y_{n+2N}}{\partial\tilde{\xi}_1} + 
    \sum^N_{i=1}\frac{\partial^2 y_{n+2N}}{\partial\tilde\xi_1\partial\xi_{2i-1}}\cdot
    \left(\frac{\Delta\xi}{2N}\right).  
  \end{equation}
  In (\ref{y-pdemag*(theta)-taylor}) and
  (\ref{y-pdemag*(theta)-diff-taylor}) the second derivatives are
  taken at some point $\theta\in[0,1]$.

  \smallskip 
  If $\tilde{\xi}_0 = \tilde{\xi}_1$ then $\xi_{2i-1} \in
  [\tilde{\xi}_1 + \Delta\xi ,\tilde{\xi}_1]$, and from Lemma
  \ref{lem:lipschitz-plus}
  \begin{equation}\label{y-pdemag-lipschitz}
    \frac{\partial y_{n+2N}}{\partial\xi_{2i-1}} = 
    \frac{\partial y_{n+2N}}{\partial\tilde{\xi}_1} +O(\Delta\xi),\;
    \mbox{if }\;\tilde{\xi}_0 = \tilde{\xi}_1.
  \end{equation}
  From the boundedness of the derivatives (\ref{deriv-bounded})
  follows that the last term in (\ref{y-pdemag*(theta)-diff-taylor})
  is $O(\Delta\xi)$.  Combining (\ref{y-pdemag*(theta)-diff-taylor}),
  (\ref{y-pdemag-lipschitz}) gives
  \begin{equation}\label{y-pdemag-diff}
    \frac{\partial y_{n+2N}}{\partial\xi_{2i-1}} = 
    \frac{\partial y_n}{\partial\tilde{\xi}_1}+O(\Delta\xi),\;
    \mbox{if }\;\tilde{\xi}_0 = \tilde{\xi}_1.
  \end{equation}

  The estimate $O(\Delta\xi)$ in
  (\ref{y-pdemag*(theta)-diff-taylor}), (\ref{y-pdemag-lipschitz}),
  and hence in (\ref{y-pdemag-diff}), does not depend on $N$.  Note
  that the last term in (\ref{y-pdemag*(theta)-taylor}) is of
  $O(\Delta\xi^2)$ order and does not depend on $N$ also, because the
  number of terms in the double sum is $N^2$. After substituting
  (\ref{y-pdemag-diff}) into (\ref{y-pdemag*(theta)-taylor}) we get
  \begin{equation}\nonumber
    y_{n+2N}(\xi_0,\xi_1,\dots,\xi_{2N},\tilde{\xi}_1,\ldots,\tilde{\xi}_n) =
    y_n(\tilde{\xi}_0 +\Delta\xi,\tilde{\xi}_1,\ldots\tilde{\xi}_n) +
    \frac{\partial y_n}{\partial\tilde{\xi}_1}\cdot
    \frac{\Delta\xi}{2} + O(\Delta\xi^2), \; 
    \mbox{if} \; \tilde{\xi}_0 = \tilde{\xi}_1.
  \end{equation}
  Here the right side including the estimate $O(\Delta\xi^2)$ does not
  depend on $N$. Due to the reachability of the demagnetized state
  according to (\ref{y-pdemag}), the right side tends to
  $y_n(\tilde{\xi}_0, \ldots, \tilde{\xi}_n)$, as
  $N\rightarrow\infty$.  Therefore, we have
  \begin{equation}\nonumber
    y_n(\tilde{\xi}_0,\ldots,\tilde{\xi}_n) = y_n(\tilde{\xi}_0 + 
    \Delta\xi, \tilde{\xi}_1,\ldots,\tilde{\xi}_n) + 
    \frac{\partial y_n}{\partial\tilde{\xi}_1}\cdot\frac{\Delta\xi}{2} 
    +O(\Delta\xi^2),\;\mbox{if} \; \tilde{\xi}_0 = \tilde{\xi}_1.
  \end{equation}
  Dividing the last equation by $\Delta\xi$ and taking the limit
  $\Delta\xi \rightarrow 0$ gives \Y2.
\end{proof}

To understand what condition \Y2 means in terms of hysteresis curves
in the \mbox{$H$-$y$} plane, let us consider two states: $(\xi_0,
\xi_0)$, which is the state on the ascending initial magnetization
curve, and $(\xi_0, \xi_1)$, which is the state on the ascending
branch of the symmetric cycle.  After the magnetic field increases by
$\delta H > 0$, the first and the second states will be $(\xi_0 +
2\delta H, 2\xi_0 + 2\delta H)$ and $(\xi_0, \xi_1 + \delta H)$
respectively (see Table \ref{tab:state-trans}).  Calculating
derivatives with respect to $\delta H$ we get in the first and the
second cases:
\begin{equation}\nonumber
  \frac{d}{d (\delta H)}y_1(\xi_0+2\delta H, \xi_0+2\delta H) = 2\frac{\partial y_1}{\partial \xi_0} + 2\frac{\partial y_1}{\partial \xi_1}, \qquad
  \frac{d}{d (\delta H)}y_1(\xi_0,\xi_1+\delta H) = \frac{\partial y_1}{\partial \xi_1}.
\end{equation}
If $\xi_1 \rightarrow \xi_0$, two curves meet each other and, as
follows from \Y2, are tangent at this point.

\medskip Condition \Y1 can be expressed in the form similar to \Y2
according to Lemma \ref{lem:Y1-diff}. Both conditions can be combined
in one
\begin{equation}\nonumber
  \mbox{\Y{1*}}\;\;\; 2^{\delta_{0i}}\frac{\partial y_n}{\partial \xi_i} + 
  \frac{\partial y_n}{\partial \xi_{i+1}} = 0,\; \mbox{if}\; 
  \xi_i = \xi_{i+1},\, n = 0,1,\ldots\,,
\end{equation}
where $\delta_{ij}$ is the Kronecker delta.

\bigskip

Consider the state $(\xi_0,\xi_1,\xi_2,\,\ldots\,,\xi_n)$.  It can be
obtained with the input $H^t\in{\cal U}^*$, as describes
Lemma~\ref{lem:r-input}.  Taking into account Corollary
\ref{cor:demag-equiv}, it is not difficult to see that if
$\xi_1=\xi_0$, the state $(\xi_0,\xi_2,\, \ldots\, ,\xi_n)$ can be
obtained with the input $-H^t$.  Usually we may interest in the
read-out functions that satisfy one of the two symmetry conditions:
\begin{align}\nonumber
&\mbox{\Y{s}}\quad y_n(\xi_0, \xi_1, \xi_2, \ldots , \xi_n) = 
y_{n-1}(\xi_0,\xi_2,\ldots,\xi_n),\;\mbox{if}\;\xi_1 = \xi_0,\; 
 n = 1,2,\ldots\,, \nonumber\\[1em]
& \mbox{\Y{a}}\quad y_n(\xi_0,\xi_1,\xi_2\ldots,\xi_n) = 
-y_{n-1}(\xi_0,\xi_2,\ldots,\xi_n),\;\mbox{if}\;\xi_1 = \xi_0\;
n = 1,2,\ldots\,.\nonumber
\end{align}
We call these functions {\em symmetric} and {\em antisymmetric}
respectively.  Antisymmetric functions can describe magnetization $M$,
\mbox{$B$-field}, and \mbox{$H$-field}, see, e.g., (\ref{H(xi)}). The
symmetric functions can describe physical values like the energy of
the system or magnetostrictive deformation.  For antisymmetric
functions $y_0(0) = 0$, because $y_1(0,0) = -y_0(0)$ due to \Y{a} and
$y_1(0,0) = y_0(0)$ due to \Y0.  It is easy to check that the
following proposition holds:
\begin{prop}
  Any functions $y_n(\xi_1, \dots, \xi_n)$, $n = 0,1\ldots\,$, can be
  expressed as the sum of its symmetric and antisymmetric parts:
  \begin{align*}
    y_n^{(s)}(\xi_0,\xi_1,\ldots,\xi_n) &= \frac{1}{2}[y_n(\xi_0,\xi_1,\ldots,\xi_n) + y_{n+1}(\xi_0,\xi_0,\xi_1,\dots,\xi_n)],\\
    y_n^{(a)}(\xi_0,\xi_1,\ldots,\xi_n) &=
    \frac{1}{2}[y_n(\xi_0,\xi_1,\ldots,\xi_n) -
    y_{n+1}(\xi_0,\xi_0,\xi_1,\dots,\xi_n)].
  \end{align*}
\end{prop}

\section{Transformations of the State Variables}\label{sec:transform}
Coordinates $\xi_0, \ldots, \xi_n$ were introduced via the values
$\Delta H_0, \ldots \Delta H_n$.  However, nothing prevented from
using, for example, the magnetization changes $\Delta M_i$ instead of
$\Delta H_i$.  This observation shows that there must be a class of
coordinate transformations which preserve conditions \Y0 -- \Y2,
\Y{s}, \Y{a}, and the state transition law.

Let us describe, how the new coordinates can be introduced with any
sequence of antisymmetric functions $u_n(\xi_0,\ldots,\xi_n)$ that
satisfies conditions \Y0 -- \Y2, \Y{a}, and the following condition:
\begin{equation}\label{u-alt-monotone}
(-1)^{n+1}\frac{\partial u_n}{\partial \xi_n} \geq  \varepsilon >0,\; 
  n = 0,1, \ldots\,,
\end{equation}
which means that $u(t)$ strictly increases (decreases) when the input
$H(t)$ increases (decreases).  The transformation between the old and
new coordinates reads
\begin{equation}\label{xi->xi'}
  \xi'_0 = \varphi_0(\xi_0),\quad \xi'_1 = \varphi_1(\xi_0, \xi_1)\,,\;
  \ldots\;,\;\xi'_n = \varphi_n(\xi_0,\ldots,\xi_n),
\end{equation}
where  
\begin{align}\label{phi-def}
  &\varphi_0(\xi_0) = - 2u_0(\xi_0),\nonumber\\
  &\varphi_k(\xi_0,\ldots,\xi_k) = (-1)^{k-1}\left[u_k(\xi_0,\ldots,\xi_k) -
    u_{k-1}(\xi_0,\ldots,\xi_{k-1})\right],
  \quad k=1,\,\ldots\,,n.
\end{align}
It can be seen that the new coordinates $\xi'_0,\,\ldots\,,\xi'_n$
are defined using the differences $\Delta u_k$ in the same way as
the coordinates $\xi_0,\,\ldots\,,\xi_n$ are defined by $\Delta H_k$;
$u_n(\xi_0,\ldots,\xi_n)$ can be expressed via $\varphi_n$
as follows:
\begin{equation}\label{u-function}
  u_n(\xi_0,\dots,\xi_n) = -\frac{1}{2}\varphi_0(\xi_0) + 
  \varphi_1(\xi_0,\xi_1) - \ldots \pm\varphi_n(\xi_0,\dots,\xi_n). 
\end{equation}  

\begin{lem}\label{lem:phi-props}
  Functions $\varphi_k(\xi_0,\ldots,\xi_k)$ defined in
  (\ref{phi-def}) satisfy the following conditions: 

\begin{enumerate}[label=(\roman*)]

\item
  If $\xi_i = 0$ then $\xi'_i = 0$, $i=0,1,\ldots\,,n$;
\item \label{phi_prop_ii}
  If $\;\xi_{i+1} = \xi_i$ then $\;\xi'_{i+1} = \xi'_i$, $i=0,1,\ldots\,,n-1$;
\item
  If $\;\xi_{i+1}=\xi_i$, $i = 1,2,\ldots\,,k-2$, $k=3,4,\ldots\,,n\;$ then
\begin{equation}
  \varphi_k(\xi_0,\ldots,\xi_i,\xi_{i+1},\ldots,\xi_k) = 
  \varphi_{k-2}(\xi_0,\ldots,\xi_{i-1},\xi_{i+2}, \ldots,\xi_k);
\end{equation}
\item
  If $\;\xi_1 = \xi_0$, $k = 2,3,\ldots,n,\;$ then
  \begin{equation}\label{phi_prop_2.1}
2\frac{\partial\varphi_1}{\partial\xi_0} + \frac{\partial\varphi_1}{\partial\xi_1} = \frac{\partial\varphi_0}{\partial\xi_0}\quad\mbox{and}\quad
    2\frac{\partial\varphi_k}{\partial\xi_0} +
    \frac{\partial\varphi_k}{\partial\xi_1} = 0;
  \end{equation}
\item 
  If $\;\xi_1=\xi_0,\; k =  1,2,\ldots,n\;$ then
  \begin{equation} \label{phi_prop_a.1}
  \varphi_k(\xi_0,\xi_1,\ldots,\xi_k) =
    \varphi_{k-1}(\xi_0,\xi_2,\ldots,\xi_k).    
  \end{equation}
\end{enumerate}
\end{lem}
\begin{proof}
  The lemma is easy to prove using conditions \Y0 -- \Y2, \Y{s}
  imposed on $u_n(\xi_0,\ldots,\xi_n)$. Item (i) follows from \Y0,
  items (ii) and (iii) from \Y1, item (iv) from \Y2, and item
  (v) from \Y{s}.
\end{proof}

The coordinate transformation (\ref{xi->xi'}) can be inverted according to the following lemma. 

\begin{lem}\label{lem:xi<->xi'}
  Coordinate transformation (\ref{xi->xi'}) is a bijection $D_n(\xi_M)
  \rightarrow D_n(\xi'_M)$, where $\xi'_M = \varphi_0(\xi_M)$.  The
  inverse transformation reads
  \begin{equation}\label{xi'->xi}
    \xi_0 = \tilde{\varphi}'_0(\xi'_0),\quad \xi_1 = 
    \tilde{\varphi}'_1(\xi'_0, \xi'_1)\,,\;
    \ldots\;,\;\xi_n = \tilde{\varphi}'_n(\xi'_0,\ldots,\xi'_n),
  \end{equation}
  and
  \begin{equation}\label{x_k<->xi'_k}
  \xi'_i = 0\quad\mbox{if and only if}\quad \xi_i = 0,\;\quad 
  \xi'_i = \xi'_{i+1}\quad\mbox{if and only if}\quad \xi_i = \xi_{i+1}.
  \end{equation}
\end{lem}
\begin{proof} 
  Due to (\ref{u-alt-monotone}), (\ref{phi-def}), for $k =
  0,1\,\ldots\,,n$ it holds $\partial{\varphi_k}/{\partial\xi_k} \geq
  \varepsilon > 0$. Obviously, $\varphi_0$ is a bijection
  $[0,\xi_M]\rightarrow [0,\varphi_0(\xi_M)]$.  According to
  Lemma~\ref{lem:phi-props}, items (i),(ii), $\varphi_k$, as a
  function of the last argument $\xi_k$, is a bijection $[0,
  \xi_{k-1}]\rightarrow [0, \xi'_{k-1}]$ for $k = 1,
  2,\,\ldots\,$. Solving equations (\ref{xi->xi'}) one-by-one gives
  (\ref{xi'->xi}) and (\ref{x_k<->xi'_k}).
\end{proof}

Here and below we mark functions expressed in the new coordinates
$\xi'_0,\ldots,\xi'_n$ by the prime symbol ($'$). In this notations
using (\ref{xi->xi'}), (\ref{xi'->xi}), we can write
\begin{equation}\label{y<->y'}
  y_n(\xi_0, \ldots, \xi_n) = 
  y_n(\tilde{\varphi}'_0(\xi'_0),\ldots,\tilde{\varphi}'_n(\xi'_0,\ldots,\xi'_n)) = 
  y'_n(\xi'_0,\ldots,\xi'_n) = y'_n(\varphi_0(\xi_0),\ldots,\varphi_n(\xi_0,\ldots,\xi_n)).
\end{equation}

\begin{prop}\label{prop:Y-invariance}
  A sequence of functions $y_n(\xi_0,\ldots,\xi_n)$, $n = 0,1,\ldots$,
  that satisfies conditions \Y0 -- \Y2, and, possibly, one of the
  symmetry condition \Y{s} or \Y{a}, satisfies the same conditions in
  the new coordinates (\ref{xi->xi'}).
\end{prop}
\begin{proof}
  For any function $y_n$ and $y'_n$ expressed in the old and in the
  new coordinates there holds
  \begin{equation}\label{y<->y'*}
    y_n(\xi_0, \ldots, \xi_n) =  y'_n(\xi'_0,\ldots,\xi'_n) = 
    y'_n(\varphi_0(\xi_0),\ldots,\varphi_n(\xi_0,\ldots,\xi_n)).
  \end{equation}

  Let $\xi'_n = 0$. Then, according to Lemma \ref{lem:xi<->xi'},
  $\xi_n = 0$ and $y_n(\xi_0, \ldots, \xi_{n-1}, \xi_n) =
  y_{n-1}(\xi_0,\ldots,\xi_{n-1})$ due to \Y0.  Because
  $y'_{n-1}(\xi'_0,\ldots,\xi'_{n-1}) =
  y_{n-1}(\xi_0,\ldots,\xi_{n-1})$ we have $ y'_n(\xi'_0, \ldots,
  \xi'_{n-1}, 0) = y'_{n-1}(\xi'_0,\ldots,\xi'_{n-1})$. Thus, \Y0 holds
  in the new coordinates.

  \smallskip Let $\xi'_{i+1} = \xi'_i$, for some $1 \leq i \leq n-1$.
  Then, according to Lemma \ref{lem:xi<->xi'}, $\xi_{i+1} = \xi_i$, and
  $y_n(\xi_0,\ldots,\xi_i,\xi_{i+1},\ldots,\xi_n) =
  y_{n-2}(\xi_0,\ldots,\xi_{i-1},\xi_{i+2},\ldots,\xi_{n-1})$ due to
  \Y1.  As can be seen from item (iii) of Lemma \ref{lem:phi-props},
  if $\xi'_{i+1} = \xi'_i$ then $\xi'_{i+2},\ldots,\xi'_n$ are the same as
  in the function $y'_{n-2}(\xi'_0,\ldots,\xi'_{i-1},\xi'_{i+2},
  \ldots,\xi'_{n-1}) = y_{n-2}(\xi_0,\ldots,\xi_{i-1},\xi_{i+2},
  \ldots,\xi_{n-1})$. Thus, \Y1 holds in the new coordinates.

  \smallskip
  Differentiating (\ref{y<->y'*}) by $\xi_1$, $\xi_0$ and
  combining the terms in the result gives
  \begin{equation}\nonumber
    2\frac{\partial y_n}{\partial\xi_0} + \frac{\partial y_n}{\partial\xi_1} = 
    2\frac{\partial y'_n}{\partial\xi'_0}\cdot\frac{\partial\varphi_0}{\partial\xi_0}+
    \sum^n_{k = 1}\frac{\partial y'_n}{\partial\xi'_k}\cdot
    \left( 2\frac{\partial\varphi_k}{\partial\xi_0}+
    \frac{\partial\varphi_k}{\partial\xi_1}\right).
  \end{equation}
  Let $\xi'_1 = \xi'_0$. Then $\xi_1 = \xi_0$ according to Lemma
  \ref{lem:xi<->xi'}, and the left side is zero due to \Y2.  Taking
  into account item (iv) of Lemma \ref{lem:phi-props}, from the above
  equation follows
  \begin{equation}\nonumber
    \frac{\partial\varphi_0}{\partial\xi_0}\cdot\left(
      2\frac{\partial y'_n}{\partial\xi'_0} + 
      \frac{\partial y'_n}{\partial\xi'_1}\right) = 0,
  \end{equation}
  Because $\partial\varphi_0/\partial\xi_0 \neq 0$, we can conclude
  that \Y2 holds in the new coordinates.

  In the similar way, using item (v) of Lemma \ref{lem:phi-props}, it
  can be found that if one of conditions \Y{s} or \Y{a} holds for
  $y_n(\xi_0,\ldots,\xi_n)$ then, respectively, $y'_n(\xi'_0,\xi'_1,\ldots,\xi'_n) =
  y'_{n-1}(\xi'_0,\xi'_2, \ldots,\xi'_n)$ or
  $y'_n(\xi'_0,\xi'_1,\ldots,\xi'_n) = -y'_{n-1}(\xi'_0,\xi'_2,
  \ldots,\xi'_n)$.
\end{proof}

\begin{prop}\label{prop:u-state-trans} Let the functions
  $u_n(\xi_0,\ldots,\xi_n)$ determine the coordinate transformation
  (\ref{xi->xi'}) according to (\ref{phi-def}).  Then the state
  transition law in the new coordinates for $\delta u$ reads exactly
  the same as it reads in the old coordinates for $\delta H$.
\end{prop}
\begin{proof}
  Let us note that the signs of $\delta u$ and $\delta H$ are always
  the same. As we can see from Table \ref{tab:state-trans}, the only
  variable that changes when $H$ changes is the last variable
  $\xi_n$. This is true with one exception, which will be considered
  separately. Due to (\ref{xi->xi'}), only $\xi'_n$ depends on
  $\xi_n$, and, according to (\ref{u-function}), $|\delta u| = \delta
  \xi'_n$.  Hence, if the state is $(\xi_0,\ldots,\xi_n)$, $n$ even and
  $\delta H < 0$, or $n$ odd and $\delta H > 0$, the new state in old
  coordinates is $(\xi_0,\ldots,\xi_n + |\delta H|)$, and in the new
  coordinates it is $(\xi'_0,\ldots,\xi'_n + |\delta u|)$.  If the
  state is $(\xi_0,\ldots,\xi_n)$, $n$ even and $\delta H > 0$, or $n$
  odd and $\delta H < 0$, the new state in old coordinates is
  $(\xi_0,\ldots,\xi_n,|\delta H|)$, and in the new coordinates it is
  $(\xi'_0,\ldots,\xi'_n, |\delta u|)$.

  The exception is $(\xi_0,\xi_0)$, $\delta H > 0$.  When the state is
  $(\xi_0,\xi_0)$ and $\delta H > 0$, the new state in the old
  coordinates is $(\xi_0 + |\delta H|,\xi_0 + |\delta H|)$. It is easy
  to check, that in the new coordinates it is $(\xi'_0 + |\delta u|,
  \xi'_0 + |\delta u|)$.
\end{proof}

According to (\ref{u-function}), (\ref{H(xi)}) and (\ref{y<->y'}), in
the new coordinates we have
\begin{align}\label{u'(xi')}
u'_n(\xi'_0,\ldots,\xi'_n) &= 
-\frac{1}{2}\xi'_0 + \xi'_1 - \ldots \pm\xi'_n,\\
H'_n(\xi'_0,\ldots,\xi'_n) &= 
-\frac{1}{2}\tilde{\varphi}'_0(\xi'_0) + \tilde{\varphi}'_1(\xi'_0,\xi'_1) - 
\ldots \pm\tilde{\varphi}'_n(\xi'_0,\ldots,\xi'_n),
\end{align}
As can be seen from (\ref{H(xi)}), $H_n(\xi_0,\ldots,\xi_n)$ satisfy
conditions \Y0 -- \Y2 and \Y{a} in coordinates
$\xi_0,\,\ldots\,,\xi_n$.  Thus, according to Proposition
\ref{prop:Y-invariance}, $H'_n(\xi'_0,\ldots,\xi'_n)$ must satisfy \Y0
-- \Y2 and \Y{a} in coordinates $\xi'_0,\,\ldots\,,\xi'_n$.  Due to
Proposition \ref{prop:u-state-trans}, the variable $u$ associated with
functions $u_n(\xi_0,\ldots,\xi_n)$ determines the state transition
law according to Table \ref{tab:state-trans}, assuming that $\delta H$
is replaced with $\delta u$, and the old coordinates are replaced with
the new ones.

\section{Algebraic Properties of Read-Out Functions}
\label{sec:yn-algebra}

Termwise operations can be performed on sequences of functions that
satisfy conditions \Y0~--~\Y2, \Y{s}, \Y{a}.  Let $f_n(\xi_0,
\dots, \xi_n)$, $g_n(\xi_0, \dots, \xi_n)$, $n = 0,1,\,\ldots$ satisfy
conditions \Y0~--~\Y2. The functions listed in the Table
\ref{table:y-alg} also satisfy \Y0~--~\Y2, and have the symmetry
as stated therein.  In the table below constants $\alpha, \beta,
\lambda$ are real numbers, $\lambda > 0$, $d\varphi(\xi)/d\xi > 0$,
$\varphi(0) = 0$, and $F$ -- any differentiable function.

\begin{table}[H]
  \renewcommand*{\arraystretch}{1.6}
  \caption{Invariance of \Y0 - \Y2 with respect to algebraic operations}
    \begin{center}
      \begin{tabular}{|c|p{10cm}|}
        \hline
        {\bf Function} & {\bf Symmetry}\\
        \hline
        $\alpha f_n + \beta g_n$ & If $u_n$ and $v_n$ have the same symmetry, the symmetry of the result is also the same\\
        \hline
        $f_n \cdot g_n$ & If $f_n$, $g_n$ have the same symmetry, the result is symmetric. If $f_n$, $g_n$ have
        the opposite symmetry, the result is antisymmetric\\
        \hline
        $f_n(\lambda \xi_0, \lambda\xi_1, \ldots, \lambda\xi_n)$ where $\lambda > 0$ & The result has the symmetry of $f_n$\\
        \hline
        $f_n(\varphi(\xi_0), \varphi(\xi_1), \ldots, \varphi(\xi_n))$ & The result has the symmetry of $f_n$\\
        \hline
        $F(f_n(\xi_0, \xi_1, \ldots, \xi_n))$ & If $f_n$ is symmetric, the result is symmetric. If $f_n$ is antisymmetric, the result is symmetric for even $F$
        and antisymmetric for odd $F$\\
        \hline
     \end{tabular}
     \label{table:y-alg}
  \end{center}
\end{table}

Let us consider in brief some examples as the illustration of the above.

\medskip

{\em Rayleigh Law}  \cite{Chikazumi1997, Bertotti1998}. Functions
\begin{equation}\label{H(xi)-phi}
H(\xi_0,\ldots,\xi_n) = -\frac{1}{2}\xi_0 + \xi_1 - \ldots \pm\xi_n
\end{equation}
satisfy \Y0 -- \Y2 and are antisymmetric.
According to the Table \ref{table:y-alg}, functions
\begin{equation}\label{M(xi)-phi}
M(\xi_0,\ldots,\xi_n) = -\frac{1}{2}\varphi(\xi_0)+\varphi(\xi_1)-
\ldots \pm\varphi(\xi_n)
\end{equation}
satisfy \Y0 -- \Y2 and are antisymmetric. This pair of functions
describe congruent multiple order reversal curves in the $H$-$M$
plane. 
Any hysteresis branch other than initial magnetization curve can be
described by equation  
\begin{equation}
\Delta M = \varphi(\Delta H),
\end{equation}
were $\Delta H$ and $\Delta M$ are absolute values of changes the field
and the magnetization relative to the reversal point. It can be seen
that for proper orientation of hysteresis loop must be
$\partial^2\varphi(\xi)/\partial \xi^2 > 0$.  The initial
magnetization curve is
\begin{equation}\label{M(H)demag-phi}
M = \pm\frac{1}{2}\varphi(\pm 2H),
\end{equation}
where ``$+$'' corresponds to ascending and ``$-$'' to descending
branches. Ascending and descending branches of symmetric hystersis cycles
can be expressed as follows:
\begin{equation}\label{M(H)sym-phi}
M \pm M_m = \pm\varphi(H_m\pm H),\quad M_m = \frac{1}{2}\varphi(2H_m),
\end{equation}
where $H_m$, $M_m$ denote the field and the magnetization in the upper
vertex of the symmetric cycle.  Letting $\varphi(\xi)= a\xi +
b\xi^2/2$ gives the well-known Rayleigh Law.

\medskip

{\em Inverse hysteresis.} In the usual case, the differential
susceptibility is positive on the hysteresis branches, so that
$M(\xi_0, \ldots ,\xi_n)$ complies with (\ref{u-alt-monotone}), and
new coordinates $\xi_0', \ldots, \xi_n'$ determined by $M(\xi_0,
\ldots ,\xi_n)$ can be introduced. As follows from Proposition
\ref{prop:u-state-trans}, the new coordinates describe the state
transitions with respect to $\delta M$ in the same way as the old ones
with respect to $\delta H$.  According to (\ref{u'(xi')}) in the new
coordinates we have
\begin{equation}\nonumber
M(\xi'_0,\ldots,\xi'_n) = -\frac{1}{2}\xi'_0 + \xi'_1 - \ldots \pm\xi'_n ,
\end{equation}
while $H(\xi'_0, \ldots, \xi'_n)$ must satisfy \Y0 -- \Y2.  One of the
possible approximations for $H(\xi'_0, \ldots, \xi'_n)$ has been
proposed in \cite{Lang2001}:
\begin{equation}\nonumber
H(\xi'_0,\ldots,\xi'_n) = a(M)\cdot\left(-\frac{1}{2}\varphi(\xi'_0) + 
\varphi(\xi'_1) - \ldots \pm \varphi(\xi'_n)\right) + b(M),
\end{equation}
where $a(M) = a\left(-\frac{1}{2}\xi'_0 + \xi'_1 - \ldots
  \pm\xi'_n\right)$ and $b(M) = b\left(-\frac{1}{2}\xi'_0 + \xi'_1 -
  \ldots \pm\xi'_n\right)$ must be even and odd functions of $M$
respectively. In this case $H(\xi'_0,\ldots,\xi'_n)$ is antisymmetric,
as it must be.

\section{Relation to the Preisach model} 

The Preisach model \cite{Preisach1935, Mayergoyz2003, DellaTorre1999,
  Bertotti1998}, see also \cite{Gorbet&all1999}, has several variants
and is considered as the most powerful hysteresis model. We examine
here only the simplest case that is usually called classical Preisah
model.

The output of the model
can be expressed via the integrals of the Preisach distribution
function over the triangles (see Fig.~\ref{fig:PM}), and, as a
function of the reduced memory sequence $H_0,\,\ldots\,,H_k$ with $H_0
< 0$, reads as follows:
\begin{equation}\label{PM-out1}
y_n(H_0, \ldots,\, H_n) - y_0(0) = - E(H_0) + E(H_0, H_1) - 
E(H_1,H_2) + \ldots \pm E(H_{n-1}, H_n),
\end{equation}
where $y_0(0)$ is the output in the demagnetized state, i.e., the
constant value, and $E(\alpha,\beta)$ is the symmetric Everett
function, $E(\alpha,\beta) = E(\beta,\alpha)$, which represents the
multiplied by 2 integral of the Preisach distribution over the
triangle $T(\alpha,\beta)$.

\begin{figure}[H]
  \begin{center}
    \includegraphics[scale=0.5]{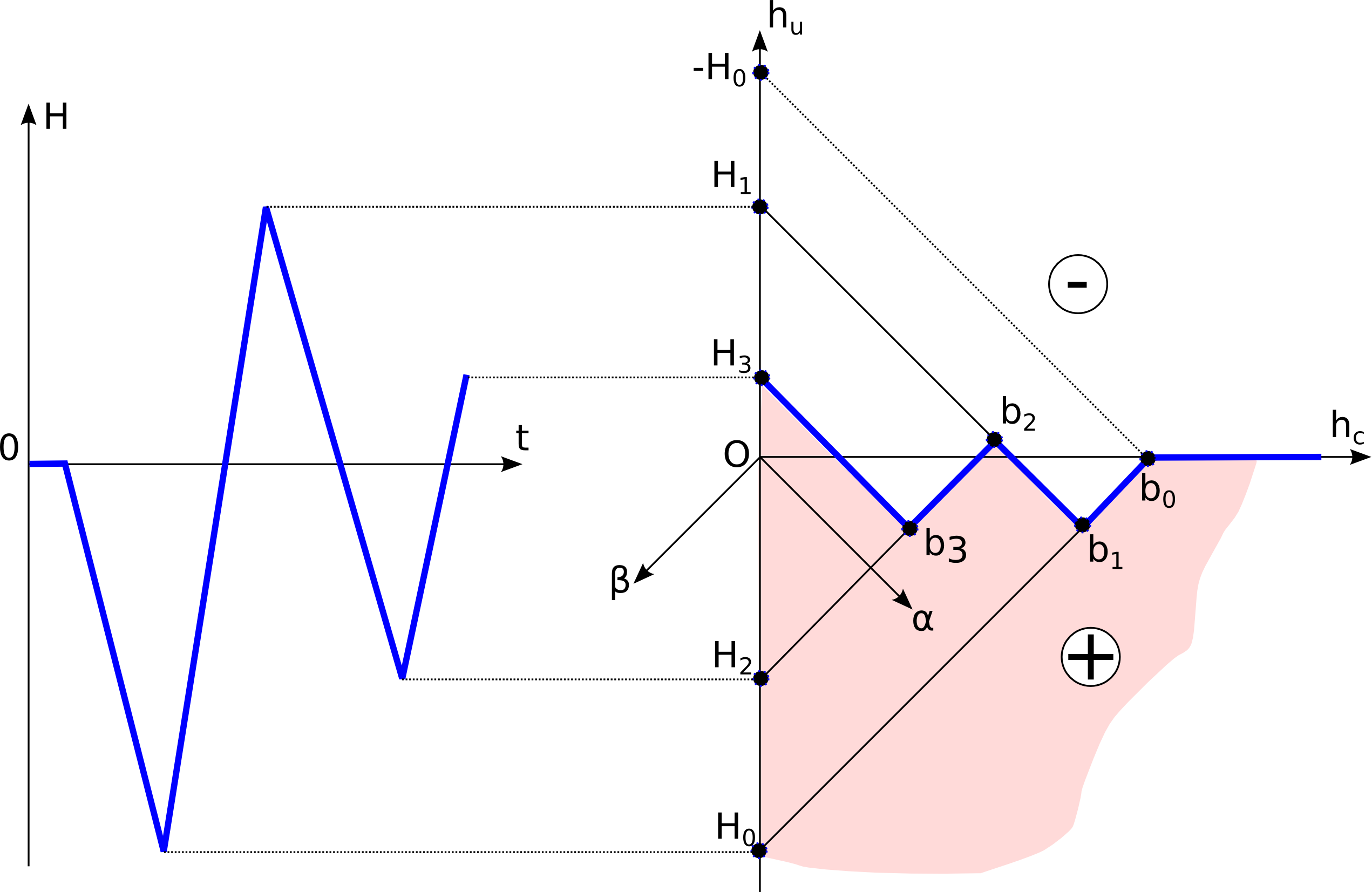}
    \caption{The input $H(t)$ applied to the Preisach model in the
      demagnetized state (left).  Preisach diagram in coordinates
      $h_c, h_u$ and the state variables $\xi_0,\,\ldots,\,\xi_n$
      (right). The line shown in bold is the boundary $h_u = b(h_c)$
      between hysterons in the positive and negative states. The
      demagnetized state corresponds to the boundary that coincides
      with the $h_c$ axis. The output of the Preisach model can be
      expressed via integrals of the Preisach distribution function
      over the ``initial magnetization triangle'' $OH_0b_0$ and the
      triangles $T(H_0,H_1) = H_0H_1b_1$, $T(H_1,H_2) = H_1H_2b_2$, $
      T(H_2,H_3) = H_2H_3b_3$.}
  \label{fig:PM}
  \end{center}
\end{figure}
 
  The equation for $E(\alpha,\beta)$ reads
\begin{equation}\label{Everett}
E(\alpha,\beta) = 2\int^\alpha_\beta\int^{\alpha'}_\beta 
\mu(\alpha',\beta')\,d\alpha'\,d\beta', \quad \mbox{where}
\quad \alpha' = h_u + h_c, \quad \beta' = h_u - h_c,
\quad \alpha \geq \beta,
\end{equation}
and $\mu(\alpha,\beta)$ is the Preisach distribution function in the
$(\alpha,\beta)$-plane.  
The symmetry of $E(\alpha,\beta)$ follows naturally from its
definition as integral over the triangle $T(\alpha,\beta)$.  It can
not be derived from (\ref{Everett}), because $\alpha \leq \beta$ imply
$\alpha' \leq \beta'$, i.e., for $\alpha \leq \beta$ the triangle
belongs to the other half-plane.

The first term in (\ref{PM-out1}), $E(H_0)$, is multiplied by 2
integral over the ``initial magnetization triangle'', like $OA_0B_0$
in Fig.~\ref{fig:PM}.  If the Preisach distribution function is
symmetric with respect to $h_c$ axis, $y_0(0)=0$ and $E(H_0) = E(H_0,
-H_0)/2$.

If the $E(\alpha,\beta)$ is known, $\mu(\alpha,\beta)$ can be found  
by differentiating (\ref{Everett}),
\begin{equation}\label{p-distr}
\mu(\alpha, \beta) = \frac{1}{2}\frac{\partial^2 E(\alpha, \beta)}
{\partial\alpha\;\partial\beta},
\quad \alpha \geq \beta.
\end{equation}

Let us consider conditions \Y0 -- \Y2 for $y_n(\xi_0,\ldots,\xi_n)$,
assuming that $H_0,\,\ldots\,,H_n$ in the right side of
(\ref{PM-out1}) are expressed via $\xi_0,\ldots,\xi_n$ according to
(\ref{xi->rmem}).  If $\xi_n = 0$ then $H_n = H_{n-1}$ and
$E(H_{n-1}, H_n) = 0$, therefore we have $y_n(\xi_0,\ldots,\xi_{n-1},0) =
y_{n-1}(\xi_0,\ldots,\xi_{n-1})$, i.e., \Y0 holds true.  Let $\xi_{k+1}=\xi_k$
for some $k$, $1\leq k\leq n-1$. This pair of variables vanishes from
all terms $E(H_{r-1}, H_r)$ such that $r\geq k$, and $E(H_{k-1},
H_k)$, $E(H_k, H_{k+1})$ become equal canceling each other out. Thus,
$y_n(\xi_0,\ldots,\xi_n) = y_{n-2}(\xi_0,\ldots,\xi_{k-1},\xi_{k+2},
\ldots \xi_n)$, which means that \Y1 holds true.

Obviously, each term $E(H_k, H_{k+1})$ satisfy \Y2 for
$k=1,\,\ldots\,,n-1$. For remaining two terms, taking into account
(\ref{xi->rmem}), we have
\begin{equation}\label{Y2-H0,H1}
\left(2\frac{\partial}{\partial\xi_0} + \frac{\partial}{\partial\xi_1}\right)
\big[-E(H_0) + H(H_0,H_1)\big] = 
\frac{\partial E(H_0)}{\partial H_0} - \frac{\partial E(H_0,H_1)}{\partial H_0}.
\end{equation}
As follows from (\ref{xi->rmem}), if $\xi_0 = \xi_1$ then $H_1 =
-H_0$, and vice versa. From the Preisach diagram in Fig \ref{fig:PM}
it can be seen directly that the right side of (\ref{Y2-H0,H1}) is
zero if $H_1 = -H_0$; this means that \Y2 holds true. 

\smallskip {\small Let $H_1 = H_0$, and $H_1$ in $E(H_0,H_1)$
  increases by the value $\delta H$. Assume that $H_0$ in $E(H_0)$
  increases by the same value.  It can be seen that both the triangles
  increase its areas by almost the same strips.  (The strips differs
  by the triangle, which area is $(\delta\xi)^2/4$). Hence, the right
  side of (\ref{Y2-H0,H1}) is zero if $\mu(\alpha,\beta)$ is bounded for
  $\alpha > \beta$.} \smallskip

Equation (\ref{Y2-H0,H1}) allows to express $E(H_0)$ via
$E(\alpha,\beta)$ and rewrite (\ref{PM-out1}) in the following form:
\begin{equation}\label{PM-out2}
y_n(H_0,\,\ldots\,,H_n) - y_0(0) = -\int_0^{H_0} \frac{\partial E(\alpha, \beta)}
{\partial \alpha}\Big|_{\beta = -\alpha}d\alpha +
E(H_0,H_1) - E(H_2,H_3) + \ldots \pm E(H_{n-1},H_n).
\end{equation}

The output (\ref{PM-out2}) satisfy conditions \Y0 -- \Y2 if
$E(\alpha,\beta)$ has partial derivatives, is symmetric with respect
to its arguments, and turns into zero when $\alpha = \beta$.  Note,
that the left and right derivatives on the line $\alpha = \beta$ are
not necessarily equal, because the points $H_k = H_{k+1}$ are newer
crossed.

As an example, let $E(\alpha,\beta) = \varphi(|\alpha - \beta|)$,
where $\varphi$ is an arbitrary smooth function such that $\varphi(0)
= 0$. Taking into account (\ref{xi->rmem}), it can be seen that
in this case (\ref{PM-out2}) turns into (\ref{M(xi)-phi}). The
Preisach distribution function determined by (\ref{p-distr}) has a
Dirac delta term:
\begin{equation}\nonumber
\mu(\alpha, \beta) =  \frac{1}{2}\varphi''(\alpha -\beta) +
\varphi'(0)\cdot\delta(\alpha -\beta),
\quad \alpha \geq \beta.
\end{equation}
This shows that the smoothness of $\mu(\alpha,\beta)$ in the classical
Preisach model is, in some sense, more strict condition than the
smoothness of $y(\xi_0,\,\ldots,\,\xi_n)$ according to
(\ref{deriv-bounded}).  Without the delta term, the initial
susceptibility is zero at the reversal points for any Preisach
distribution function.  This effect is eliminated in the moving
Preisach model by introducing the mean field interaction.

Another element of the Preisach model that can be expressed via
coordinates $\xi_0,\,\ldots\,,\xi_n$, is the staircase boundary
$b(h_c)$ between the Preisach units in $+1$ and $-1$ states. Taking
into account (\ref{xi-def}), it can be seen that the $h_c$-coordinates
of points $b_0,\,\ldots\,,b_3$ are $\xi_0/2,\,\ldots\,,\xi_3/2$
correspondingly, and the state evolution rules in $\xi$-coordinates
are the same as presented in Table \ref{tab:state-trans}. Also, it
is not difficult to obtain the following equation for the boundary:
\begin{equation}\nonumber
  b(h_c) = \int^\infty_{2h_c}\left[-\frac{1}{2}e(\xi_0 - \xi) + e(\xi_1 - \xi) - 
 e(\xi_2 - \xi) + \ldots \pm e(\xi_n - \xi)\right]d\xi,
\end{equation}
where $e$ denotes the Heaviside step function. Two states of Preisach
model differ in the number of the hysterons enclosed between two
boundaries $b(h_c)$.  Thus, the distance between the boundaries can be
considered as a distance between states of the Preisach
model. Probably, we can use the same metric for the states
parametrized by coordinates $\xi_0,\,\ldots,\,\xi_n$ irrespective of
the Preisach model. For example, it can be a metric induced by the
following norm:
\begin{equation}\nonumber
\max_{\xi\in[0,\xi_M]} \left|\int^{\xi_M}_\xi
\left[-\frac{1}{2}e(\xi_0 - \xi') + e(\xi_1 - \xi') - 
 e(\xi_2 - \xi') + \ldots \pm e(\xi_n - \xi')\right]d\xi'\right|,
\end{equation}
where the integrands are considered as a subset in the linear space of
integrable functions.

As it can be seen from the above, the state-space of the Preisach
model can be parametrized by the variables
$\xi_0,\xi_1\,\ldots\,,\xi_n$. All conditions \Y0 -- \Y2 hold true for
the classical Preisach model, as it must be, because the model
exhibits the return point memory and has the reachable demagnetized
state. However, the Preisach model uses a definite form of the
read-out functions (\ref{PM-out1}), which is, from the point of view
studied in this article, a special case.

\section{Conclusions}

The definition of the input-output system includes a set of admissible
inputs ${\cal U}$ and a set of output variables~$Y$. These two sets
determine the experiments that can be used to study the behavior of
the system put into a fixed initial state before each experiment. If
the input equivalence is known, the state-space and the state
transition law can be established without modeling the internal
structure of the system. This approach is applicable to the
deterministic stationary input-output systems, not necessary
hysteretic (see Appendix~\ref{app:state-space}).  This method does not
take into account energy conservation law or any other thermodynamic
restrictions. In this aspect, the description based on the input
equivalence is similar to kinematics.

\medskip

In the case of the scalar hysteresis, the input equivalence is
determined by the return point memory and rate independence, as
described in Section \ref{sec:RPM}.  Coordinates $\xi_0,\xi_1,
\,\ldots\,, \xi_n$ in the state-space are introduced via the
differences $\Delta H_i$ along the hysteresis branches.  The
consistency with the return point memory and the reachability of the
demagnetized state provide necessary and sufficient conditions \Y0 --
\Y2 on the read-out functions $y_n(\xi_0,\xi_1\, \ldots\,,\xi_n)$.

The coordinate transformations allow to switch to coordinates
$\xi'_0,\,\ldots\,,\xi'_n$ that have the same properties with respect
to a different input, e.g., $M(t)$ instead of $H(t)$. In this way, a
direct and a reverse ferromagnetic hysteresis can be expressed in the
similar form.  A set of termwise operations can be performed on the
sequences of functions $y_n(\xi_0,\,\ldots\,,\xi_n)$ without violating
conditions \Y0 -- \Y2. Together with the coordinate transformations,
it can be used as a tool for building different approximations of
multiple order reversal curves. Though the consideration is made in
the framework of magnetic hysteresis, the results, probably, can be
applied or adapted to other manifestations of hysteresis.

\bibliography{hyst}{}
\bibliographystyle{plain}

\appendix
\section{Appendix: State-Space of the Input-Output Systems
}\label{app:state-space}

This Appendix describes how the state-space and the state transition
law can be established for any stationary input-output system,
assuming that the input equivalence is known from the experiments. Its
aim is to provide a background to the consideration presented in the
main part of the article. The starting point is a definition of the
input-output system close to one given in \cite{Willems1972}.

\begin{dfn}
  \label{dfn:io-system}
  An input-output system $\Sigma = \Sigma(U,{\cal U}, Y, {\cal Y}, F)$
  is a collection of a set $U$ of {\em input variables}, an {\em input
    space} ${\cal U}$ of $U$-valued functions on $\mathbb{R}$, a set
  $Y$ of {\em output variables}, an {\em output space} ${\cal Y}$ of
  $Y$-valued functions on $\mathbb{R}$, and an {\em input-output map}
  $F$ from ${\cal U}$ into ${\cal Y}$, such that the following
  conditions are met:
\begin{enumerate}[label=(\roman*)]
\item For any $T \in\mathbb{R}$, if $u(t) \in{\cal U},\; y(t) \in
  {\cal Y}$, then $u(t + T)\in {\cal U},\; y(t + T) \in{\cal Y}$;
\item For any $u_1, u_2 \in {\cal U}$, if $u_1(t') = u_2(t')$ for all
  $t'\leq t$, then $F[u_1](t') = F[u_2](t')$ for all $t' \leq t$;
\item For any $u_1, u_2 \in {\cal U}$, $T \in \mathbb{R}$, if $u_2(t)
  = u_1(t + T)$ for all $t\in\mathbb{R}$, then $F[u_2](t) = F[u_1](t +
  T)$\\ for all $t\in\mathbb{R}$;
\item For any $u\in{\cal U}$, there exists a value $t_0(u)$ such that
  $u(t) = u_0$ for all $t\leq t_0$. The value $u_0$ is the same for
  all $u \in{\cal U}$.
\end{enumerate}
\end{dfn}

We interpret the definition of $\Sigma$ as a description of the
system's behavior in a class of experiments specified by the set
of admissible inputs ${\cal U}$ and the set of measured output
variables $Y$.  The properties of the system itself are represented by
the map $F$.  In the above definition item (i) means that the spaces
${\cal U}, {\cal Y}$ are closed under the shift operator; items (ii)
and (iii) mean that the system assumed to be {\em deterministic} and
{\em stationary}. Determinism reflects the fact that the input
completely determines the output, and the future can not influence the
past.  Stationarity means that the experiments with the system can be
performed at any time with the same result.

Because it is impossible to start experiment at $t = -\infty$, it is
assumed that the system is put into the same initial state $x_0$ at
the beginning $t_0$ of each experiment, where $x_0$ is a steady state
under the constant input $u_0$. The term ``put into the initial
state'' means that we do something with the system that makes the
output be uniquely determined by the input. The system must not change
while $u(t) = u_0$, and hence we can let $u(t) = u_0$ for $t
\leq t_0(u)$ according to item (iv) of the
definition.

\begin{dfn}\label{dfn:prolong}
  The restriction of the input $u$ to the interval $(-\infty,t)$ is
  designated as $u^t$ and called {\em input} along with the inputs
  defined on the interval $(-\infty, \infty)$. The set of all inputs
  $u^t$ is designated as ${\cal U}^t$, and the union $\underset{t\in
    \mathbb{R}}\cup\;{\cal U}^t$ is designated as ${\cal U}^*$. Any
  input $\tilde{u}\in {\cal U}$ such that $\tilde{u}^t = u^t$ is
  called {\em prolongation} of the input $u^t$.
\end{dfn}
We denote the restriction of an input $u\in{\cal U}$ to the interval
$[t_1,t_2)$ as $u^{[t_1,t_2)}$, and concatenation of two restricted
inputs as $u^{[t_1,t_2)}\vee u^{[t_2,t_3)}$. Notations for the outputs
are defined in the similar way.

\medskip {\small Note that the point $t$ does not belong to the domain
  of $u^t$.  Inputs and outputs are not necessary assumed to be
  continuous functions of time. Thus, we must take care for the time
  intervals not to overlap.  For continuous inputs and outputs instead
  of $(-\infty, t)$, $[t_1,t_2)$ can be used $(-\infty,t]$ and
  $[t_1,t_2]$, see also Proposition \ref{uy-read-out}.}  \medskip

Let us define {\em equivalent inputs} as the inputs from
${\cal U}^t$ that are not distinguishable after the end time $t$ in
the experiments with the system as follows.

\begin{dfn}
  \label{dfn:tequiv}
  We call inputs $u_1^t, u_2^t \in {\cal U}^t$ {\em equivalent} and
  write $u_1^t\sim u_2^t$ if both two conditions are met:
  \begin{enumerate}[label=(\roman*)]
  \item For any prolongation $\tilde{u}_1$ of $u_1^t$ there exists a
    prolongation $\tilde{u}_2$ of $u_2^t$, and vice versa, such that\\
    $\tilde{u}_1^{[t,\infty)} = \tilde{u}_2^{[t,\infty)}$;
  \item For any prolongations $\tilde{u}_1,\,\tilde{u}_2$ of
    $u_1^t,\,u_2^t$, such that $\tilde{u}_1^{[t,\infty)} =
    \tilde{u}_2^{[t,\infty)}$, there holds $\tilde{y}_1^{[t,\infty)}
    = \tilde{y}_2^{[t,\infty)}$, where $\tilde{y}_1 =
    F[\tilde{u}_1],\; \tilde{y}_2 = F[\tilde{u}_2]$
  \end{enumerate}
\end{dfn}

The item (i) means that the equivalent inputs can be always compared
using the prolongations that are the same after the end time~$t$.
According to the item (ii), such prolongations must yield the same
outputs after~$t$. The binary relation ``$\sim$'' on ${\cal U}^t$ is
the {\em equivalence relation} according to the following lemma.

\begin{lem}
  \label{lem:tequiv}
  The binary relation on the set of inputs ${\cal U}^t$ introduced by
  Definition \ref{dfn:tequiv} is reflexive ($u^t \sim u^t$),
  symmetric (if $u_1^t \sim u_2^t$ then $u_2^t \sim u_1^t$), and
  transitive (if $u_1^t \sim u_2^t$ and $u_2^t \sim u_3^t$ then $u_1^t
  \sim u_3^t$).
\end{lem}
\begin{proof}
  The reflexivity and the symmetry are obvious. To prove the
  transitivity, let $u_1^t\sim u_2^t$ and $u_2^t\sim u_3^t$. According
  to Definition \ref{dfn:tequiv}, for any prolongation $\tilde{u}_1$
  there exists the prolongation $\tilde{u}_2$ such that
  $\tilde{u}_1^{[t,\infty)} = \tilde{u}_2^{[t,\infty)}$. Similarly,
  for prolongation $\tilde{u}_2$ must exist $\tilde{u}_3$, such that
  $\tilde{u}_2^{[t,\infty)} = \tilde{u}_3^{[t,\infty)}$. Thus,
  $\tilde{u}_1^{[t,\infty)} = \tilde{u}_3^{[t,\infty)}$, and the first
  condition of Definition \ref{dfn:tequiv} is met.
  
  Suppose now that for prolongations $\tilde{u}_1$ of $u_1^t$ and
  $\tilde{u}_3$ of $u_3^t$ there holds $\tilde{u}_1^{[t,\infty)} =
  \tilde{u}_3^{[t,\infty)}$. Because $u_1^t \sim u_2^t$ and $u_2^t
  \sim u_3^t$ there exist prolongations $\tilde{u}_{2'}$ and
  $\tilde{u}_{2''}$ of $u_2^t$ such that $\tilde{u}_{2'}^{[t,\infty)}
  = \tilde{u}_1^{[t,\infty)}$ and $\tilde{u}_{2''}^{[t,\infty)} =
  \tilde{u}_3^{[t,\infty)}$.  However,
  $\tilde{u}_1^{[t,\infty)}=\tilde{u}_3^{[t,\infty)}$, which means
  that $\tilde{u}_{2'} = \tilde{u}_{2''}$. Due to $u_1^t \sim u_2^t$
  and $u_2^t \sim u_3^t$ we have $\tilde{y}_1^{[t,\infty)} =
  \tilde{y}_{2'}^{[t,\infty)}$ and $\tilde{y}_3^{[t,\infty)} =
  \tilde{y}_{2''}^{[t,\infty)}$. Hence, $\tilde{y}_1^{[t,\infty)} =
  \tilde{y}_3^{[t,\infty)}$. 
\end{proof}

With Definition \ref{dfn:tequiv}, the inputs $u^t,\,v^t\in {\cal U}^t$,
i.e. the inputs ending at the same time $t$, can be compared.  To
compare inputs $u_1^{t_1},\,u_2^{t_2}\in {\cal U}^*$, we can shift one
of them or both to make the end times equal and compare them using
Definition \ref{dfn:tequiv}. Let us introduce the shift operator
$\sigma^T$ on ${\cal U}^*$ that acts on the input $u^t$ as follows:
$\sigma^Tu^t(t') = u^{t+T}(t'-T)$. Of course, usually $\sigma^Tu^t \neq
u^{t+T}$, however, we can write $\sigma^{T_1}\sigma^{T_2} =
\sigma^{T_1+T_2}$ and $\sigma^T{\cal U}^t = {\cal U}^{t+T}$.

\begin{lem}
  \label{lem:sigma-equiv}
  The equivalence relations on the sets of inputs ${\cal U}^t$ are
  invariant under the shift operator, i.e., if $u_1^t\sim u_2^t$, then for
  any $T\in\mathbb{R}$ there holds $\sigma^Tu_1^t\sim \sigma^Tu_2^t$.
\end{lem}
\begin{proof}
  The proof is straightforward, taking into account the stationarity
  condition according to Definition~\ref{dfn:io-system}.
\end{proof}

\noindent
Using Lemma \ref{lem:sigma-equiv}, we can extend the input equivalence
from sets ${\cal U}^t$ to the set ${\cal U}^*$ as defined below.

\begin{dfn}
  \label{dfn:*equiv} Inputs $u_1^{t_1}, u_2^{t_2}\in{\cal U}^*$ we
  call equivalent and write $u_1^{t_1} \sim u_2^{t_2}$ if and only if
  the inputs $\sigma^{t_1 - t_2} u_2^{t_2}$ and $u_1^{t_1}$ are
  equivalent in the sense of Definition~\ref{dfn:tequiv}.
\end{dfn}
\begin{prop}
  The binary relation introduced by Definition \ref{dfn:tequiv} and
  extended to the set ${\cal U}^*$ according to Definition
  \ref{dfn:*equiv} is reflexive, symmetric, and transitive, i.e., is
  the equivalence relation.
  \label{prop:*equiv}
\end{prop}
\begin{proof} The reflexivity follows directly from Lemma
  \ref{lem:tequiv}. The symmetry can be proved by applying
  $\sigma^{t_2-t_1}$ to the both sides of the relation $\sigma^{t_1 - t_2}
  u_2^{t_2} \sim u_1^{t_1}$ and using Lemma \ref{lem:sigma-equiv}.
  To prove the transitivity, assume that $u_1^{t_1}
  \sim\sigma^{t_1-t_2} u_2^{t_2}\,, u_2^{t_2}\sim
  \sigma^{t_2-t_3}u_3^{t_3}$. By applying $\sigma^{t_1-t_2}$ to the
  both sides of the second equivalence and using the transitivity of
  ``$\sim$'' on ${\cal U}^{t_1}$, we get
  $u_1^{t_1}\sim\sigma^{t_1-t_3}u_3^{t_3}$, which proves the
  statement.
\end{proof}

It is well known that an equivalence relation partitions the set into
{\em equivalence classes}, so that each element of the set belongs to
one and only one class, and all elements in one class are
equivalent. {\em Projection map} gives the equivalence class by the
element of the set. The projection map is determined by the
equivalence relation and, vice versa, the equivalence relation is
determined by the projection map.

\begin{dfn}
  \label{dfn:states}
  Let ``$\sim$'' be the equivalence relation on ${\cal U}^*$ according to
  Definition \ref{dfn:*equiv}. Then the set of equivalence classes $X
  = {\cal U}^*\!/\!\!\sim$ is called {\em (minimal) space of states}
  of the system $\Sigma$. The corresponding projection map is
  designated as $\Gamma:{\cal U}^*\rightarrow X$.
\end{dfn}

Giving the name ``state'' to equivalence class $x\in X$ is approved by the
following consideration.

\bigskip

According to the definition of equivalence classes, $u_1^{t_1}\sim
u_2^{t_2}$ is the same as $u_1^{t_1},u_2^{t_2}\in x$ for some $x\in
X$, and $\Gamma(u_1^{t_1}) = \Gamma(u_2^{t_2})$ if and only if
$u_1^{t_1} \sim u_2^{t_2}$.  For each input $u^t \in {\cal U}^*$ there
exists one and only one equivalence class $x$ such that $u^t\in x$. We
can write
\begin{equation}\label{Gamma}
x(t) = \Gamma(u^t),
\end{equation}
where the state $x$ is associated with the end time of the input
$u^t$. This means that if the system evolves under some input $u\in
{\cal U}$, $x(t)$ shows its state at each time instance $t$. The
``initial state'' $x_0$ introduced earlier to establish the connection
between a physical system and its mathematical description, is, in
terms of the equivalence classes, the set of constant inputs such that
$u^t(t') = u_0$ for all $t'\in (-\infty, t)$. For any $u \in {\cal
  U}$, according to item (iv) of Definition \ref{dfn:io-system} we
have $x(t) = x_0$ if $t < t_0(u)$.

\medskip
Let us denote as ${\cal U}^{[t,t+s)}(x)$, $s>0$, the set of inputs $u\in
{\cal U}$ restricted to the interval $[t,t+s)$, such that $u^t\in x$.

\begin{prop}
  \label{prop:state-trans}
  The projection map $\Gamma$ determines the state transition function
  $\phi: X\times {\cal U}^{[t,t+s)}(x) \rightarrow X$, such
  that for any $u\in {\cal U}$ and any $s> 0$ there holds
  \begin{equation}\label{state-trans}
    x(t+s) = \phi\,(x(t), u^{[t,t+s)}).
  \end{equation}
\end{prop}

\begin{lem}\label{lem:sprolong} For any $u_1,u_2\in{\cal U}$  and $s>0$,
  if $u_1^t\sim u_2^t$, then $u_1^t\vee u_2^{[t,t+s)}\in {\cal
    U}^{t+s}$ and $u_2^t\vee u_1^{[t,t+s)}\in {\cal U}^{t+s}$.  If, in
  addition, $u_1^{[t,t+s)} = u_2^{[t,t+s)}$ then $u_1^{t+s} \sim
  u_2^{t+s}$.
\end{lem}
\begin{proof}
  The first part of the statement follows directly from item (i) of
  Definition \ref{dfn:tequiv}. To prove the second part, let us
  consider a prolongation $\tilde{u}_1$ of $u_1^{t+s}$.  Obviously,
  $\tilde{u}_1$ is also a prolongation of $u_1^t$.  Because $u_2^t\sim
  u_1^t$, there exists a prolongation $\tilde{u}_2$ of $u_2^t$, such
  that $\tilde{u}_2^{[t,\infty)} = \tilde{u}_1^{[t,\infty)}$, and
  $\tilde{y}_2^{[t,\infty)}=\tilde{y}_1^{[t,\infty)}$.  This
  proves the lemma, because $\tilde{u}_2^{[t,t+s)}=u_1^{[t,t+s)} =
  u_2^{[t,t+s)}$ and hence $\tilde{u}_2$ is a prolongation of
  $u_2^{t+s}$.
\end{proof}

\begin{proof}[Proof of Proposition \ref{prop:state-trans}]
  From (\ref{Gamma}), we have $x(t) = \Gamma(u^t)$, $x(t+s) =
  \Gamma(u^{t+s}) = \Gamma(u^t\vee u^{[t,t+s)})$. Due to Lemma
  \ref{lem:sprolong}, for any $v^t\sim u^t$, such that $v^{[t,t+s)} =
  u^{[t,t+s)}$, it holds $v^{t+s}\sim u^{t+s}$, and we can conclude
  that $\Gamma(u^t\vee u^{[t,t+s)})$ depends on $\Gamma(u^t)$, but not
  on $u^t$ directly, thus we can express $x(t+s)$ according to
  (\ref{state-trans}).
\end{proof}

\begin{prop}
  \label{prop:semi-group}
  The state transition function comply with the {\em semi-group property}
  \begin{equation}\label{semi-group}
    \phi\,(x(t), u^{[t,t+s)}) = \phi\,(\phi\,(x(t), u^{[t,t+s')}),u^{[t+s',t+s)}),\;
    s \geq s'\geq 0.
  \end{equation}
\end{prop}
\begin{proof}
As in the proof of Proposition \ref{prop:state-trans} we have
\begin{align*}
&x(t+s') = \Gamma(u^t\vee u^{[t,t+s')}) = \phi\,(x(t), u^{[t,t+s')}),\\
&\Gamma(u^t\vee u^{[t,t+s)}) = \phi\,(x(t), u^{[t,t+s)}),\\
&\Gamma(u^{t+s'}\vee u^{[t+s',t+s)}) = \phi\,(x(t+s'), u^{[t+s',t+s)}).
\end{align*}
Because $u^t\vee u^{[t,t+s)} = u^{t+s'}\vee u^{[t+s',t+s)}$,
the above equations give (\ref{semi-group}).
\end{proof}

\begin{lem}
  \label{lem:read-out}
  Let $u_1^t, u_2^t\in x(t)$, $y_1 = F[u_1],\,y_2 = F[u_2]$.  If
  $u_1(t) = u_2(t)$ then $y_1(t) = y_2(t)$.
\end{lem}
\begin{proof}
  Let $\tilde{u}_2$ be a prolongation of $u_2^t$, such that
  $\tilde{u}_2^{[t,\infty)}= u_1^{[t,\infty)}$. It exists, because
  $u_1(t)\sim u_2(t)$, and $\tilde{y}_2^{[t,\infty)}=
  y_1^{[t,\infty)}$ holds true. Thus, $y_1(t) =
  \tilde{y_2}(t)$. Because $\tilde{u}_2(t) = u_1(t)$, and $u_2(t) =
  u_1(t)$, we have $\tilde{u}_2^{(-\infty,t]} =
  u_2^{(-\infty,t]}$. According to item (ii) of Definition
  \ref{dfn:io-system}, from $\tilde{u}_2^{(-\infty,t]} =
  u_2^{(-\infty,t]}$ follows $\tilde{y}_2^{(-\infty,t]} =
  y_2^{(-\infty,t]}$. Hence, $y_2(t) = \tilde{y}_2(t)$.
\end{proof}

\begin{prop}
  \label{prop:read-out}
  For any $u\in{\cal U}$ the output $y(t)$ is determined by the state
  $x(t)$ and the output value $u(t)$,
\begin{equation}\label{y-read-out}
  y(t) =  f(x(t),u(t)).
\end{equation}
\end{prop}
\begin{proof}
  Because the system is deterministic, the output $y^{(-\infty,t]}$
  depends only on $u^{(-\infty,t]} = u^t\vee u(t)$. According to Lemma
  \ref{lem:read-out}, $y(t)$ depends on $x(t)$, but not on $u^t\in
  x(t)$ directly. Thus, (\ref{y-read-out}) holds true.
\end{proof}

\noindent
Proposition \ref{prop:read-out} can be strengthened assuming that the
inputs are continuous functions of time.

\begin{prop}\label{continuous_inputs}
  Let $U$ be a metric space and all inputs $u\in{\cal U}$ be
  continuous functions of time.  Then from $u_1^t \sim u_2^t$ follows
  that $u_1(t) = u_2(t)$, and
\begin{equation}\label{uy-read-out}
  y(t) = f(x(t)),\; u(t) = g(x(t)).
\end{equation}
\end{prop}
\begin{proof}
  Because $\underset{t'\rightarrow t}\lim u^t(t') = u(t)$, the
  inequality $u_1(t) \neq u_2(t)$ and the equivalence $u_1^t \sim
  u_2^t$ contradict each other, due to item (i) of Definition
  \ref{dfn:tequiv}. Thus, for all $u^t\in x(t)$ the value $u(t)$ must be
  the same, which proves the proposition.
\end{proof}

\end{document}